\documentclass[12pt]{article}
\usepackage{tikz}
\usepackage{tabularx}
\usepackage{float}
\usepackage{makecell}
\usepackage{mathrsfs}
\usepackage{amsmath,amsfonts,amssymb,rotating,amsthm}
\usepackage{hyperref}
\usepackage{array}
\usepackage{breqn}
\usepackage{setspace}
\allowdisplaybreaks
\makeatletter
\renewcommand{\maketag@@@}[1]{\hbox{\m@th\normalsize\normalfont#1}}%
\makeatother
\def\qed{\nopagebreak\hfill{\rule{4pt}{7pt}}}

\newtheorem{theo}{Theorem}[section]

\newtheorem{lemm}[theo]{Lemma}

\newtheorem{conj}[theo]{Conjecture}
\theoremstyle{remark}

\numberwithin{equation}{section}

\newdimen\Squaresize \Squaresize=11pt
\newdimen\Thickness \Thickness=0.7pt
\def\Square#1{\hbox{\vrule width \Thickness
   \vbox to \Squaresize{\hrule height \Thickness\vss
    \hbox to \Squaresize{\hss#1\hss}
   \vss\hrule height\Thickness}
\unskip\vrule width \Thickness} \kern-\Thickness}

\def\Vsquare#1{\vbox{\Square{$#1$}}\kern-\Thickness}

\def\moins{\raise 1pt\hbox{{$\scriptstyle -$}}}
\begin{document}

\begin{center}
{\Large \bf  Laguerre inequality and determinantal inequality for the broken $k$-diamond partition function}
\end{center}

\begin{center}
Eve Y.Y. Yang\\[8pt]
Department of Mathematics\\
Tianjin University\\
Tianjin 300072, P. R. China\\[6pt]
Email: {\tt yangyaoyao@tju.edu.cn\tt }
\end{center}

\vspace{0.3cm} \noindent{\bf Abstract.} In 2007, Andrews and Paule introduced the broken $k$-diamond partition function $\Delta_{k}(n)$, which has received a lot of researches on the arithmetic propertises. In this paper, we will prove the broken $k$-diamond partition function satisfies the Laguerre inequalities of order $2$ and the determinantal inequalities of order $3$ for $k=1$ or $2$. Moreover, we conjectured the thresholds for the Laguerre inequalities of order $m$ and the positivity of $m$-order determinants for $4\leq m\leq 14$ for the broken $k$-diamond partition function when $k=1$ or $2$.\\

\section{Introduction}
The main purpose of this paper is to prove the Laguerre inequality of order $2$ and the determinantal inequality of order $3$ for the broken $k$-diamond partition function. The notion of broken $k$-diamond partitions was introduced by Andrews and Paule \cite{Andrews}. A broken $k$-diamond partition $\pi=(b_2,\ldots,b_{2k+2},\ldots,$ $b_{(2k+1)l+1};a_1,\ldots,a_{2k+2},\ldots a_{(2k+1)l+1})$ is a plane partition satisfying the relations illustrated in Figure \ref{fig-broken-k}, where $a_i$, $b_i$ are non-negative integers and $a_i\rightarrow a_j$ means $a_i\geq a_j$. More precisely, each building block in Figure \ref{fig-broken-k}, except for the broken block $(b_2,b_3,\ldots b_{2k+2})$ has the same order structure as shown in Figure \ref{fig-k-elo-1}. We call each block a $k$-elongated partition diamond of length 1, or a $k$-elongated diamond, for short. It should be noted that the broken block $(b_2,b_3,\ldots b_{2k+2})$ is
also a $k$-elongated partition diamond of length 1 from which a source $b_1$ is deleted.
	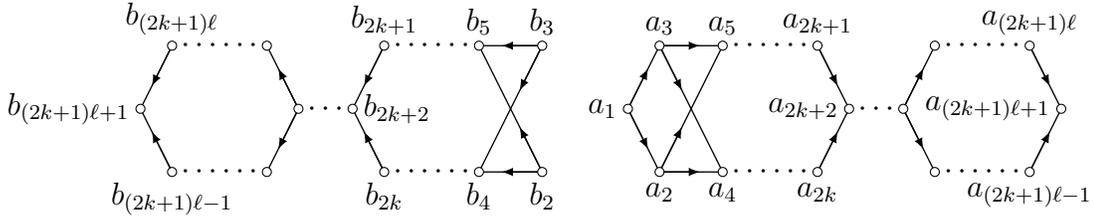
\begin{figure}
	\begin{center}
		\begin{tikzpicture}[scale=0.6]	
			
			\coordinate [label=left:$a_{1}$] (a1) at (2.1,0);
			\draw (a1) circle(.1);
			
			\coordinate[label=above:$a_{5}$] (a5) at (4.2,1.4);
			\draw (a5) circle(.1);
			
			\coordinate [label=below:$a_{4}$] (a4) at (4.2,-1.4);
			\draw (a4) circle(.1);
			
			\coordinate[label=above:$a_{3}$] (a3) at (2.8,1.4);
			\draw (a3) circle(.1);
			
			\coordinate[label=below:$a_{2}$] (a2) at (2.8,-1.4);
			\draw (a2) circle(.1);
			
			\draw[line width=0.5pt] (2.9,1.4)--(4.1,1.4);
			\draw[-latex,line width=0.5pt] (2.9,1.4)--(3.7,1.4);
			
			\draw[line width=0.5pt] (2.9,-1.4)--(4.1,-1.4);
			\draw[-latex,line width=0.5pt] (2.9,-1.4)--(3.7,-1.4);
			
			\node at (4.5,1.4) {$\cdot$};	
			\node at (4.8,1.4) {$\cdot$};	
			\node at (5.1,1.4) {$\cdot$};	
			\node at (5.4,1.4) {$\cdot$};	
			\node at (5.7,1.4) {$\cdot$};	
			\node at (6,1.4) {$\cdot$};
				
			\coordinate [label=above:$a_{2k+1}$] (a2k+1) at (6.3,1.4);
			\draw (a2k+1) circle(.1);
			
			\coordinate [label=below:$a_{2k}$] (a2k) at (6.3,-1.4);
			\draw (a2k) circle(.1);
			
			\node at (4.5,-1.4) {$\cdot$};	
			\node at (4.8,-1.4) {$\cdot$};	
			\node at (5.1,-1.4) {$\cdot$};	
			\node at (5.4,-1.4) {$\cdot$};	
			\node at (5.7,-1.4) {$\cdot$};	
			\node at (6,-1.4) {$\cdot$};
			
			\coordinate [label=left:$a_{2k+2}$] (a2k+2) at (7,0);
			\draw (a2k+2) circle(.1);

			\draw[line width=0.5pt] (2.845,-1.31)--(4.155,1.31);
			\draw[-latex,line width=0.5pt] (2.845,-1.31)--(3.3,-0.4);

			\draw[line width=0.5pt] (2.845,1.31)--(4.155,-1.31);
			\draw[-latex,line width=0.5pt] (2.845,1.31)--(3.3,0.4);
			
			\draw[-latex,line width=0.5pt] (2.145,0.09)--(2.55,0.9);
			\draw[line width=0.5pt] (2.145,0.09)--(2.755,1.31);
			
		    \draw[-latex,line width=0.5pt] (2.145,-0.09)--(2.55,-0.9);
			\draw[line width=0.5pt] (2.145,-0.09)--(2.755,-1.31);
			
			\draw[line width=0.5pt] (6.345,1.31)--(6.955,0.09);
			\draw[-latex,line width=0.5pt] (6.345,1.31)--(6.75,0.5);
			
			\draw[-latex,line width=0.5pt] (6.345,-1.31)--(6.75,-0.5);
			\draw[line width=0.5pt] (6.345,-1.31)--(6.955,-0.09);
			
			\node at (7.3,0) {$\cdot$};	
			\node at (7.6,0) {$\cdot$};	
			\node at (7.9,0) {$\cdot$};

			\draw (8.2,0) circle(.1);

			\draw (8.9,1.4) circle(.1);
			
			\draw (8.9,-1.4) circle(.1);

			\node at (9.2,1.4) {$\cdot$};	
			\node at (9.5,1.4) {$\cdot$};	
			\node at (9.8,1.4) {$\cdot$};	
			\node at (10.1,1.4) {$\cdot$};	
			\node at (10.4,1.4) {$\cdot$};	
			\node at (10.7,1.4) {$\cdot$};
				
			\coordinate [label=above:$a_{(2k+1)\ell}$] (a2k+1n) at (11,1.4);
			\draw (a2k+1n) circle(.1);
			
			\coordinate [label=below:$a_{(2k+1)\ell-1}$] (a2k+1n-1) at (11,-1.4);
			\draw (a2k+1n-1) circle(.1);
			
			\node at (9.2,-1.4) {$\cdot$};	
			\node at (9.5,-1.4) {$\cdot$};	
			\node at (9.8,-1.4) {$\cdot$};	
			\node at (10.1,-1.4) {$\cdot$};	
			\node at (10.4,-1.4) {$\cdot$};	
			\node at (10.7,-1.4) {$\cdot$};
			
			\coordinate [label=left:$a_{(2k+1)\ell+1}$] (a2k+1n+1) at (11.7,0);
			\draw (a2k+1n+1) circle(.1);

			\draw[-latex,line width=0.5pt] (8.245,0.09)--(8.65,0.9);
			\draw[line width=0.5pt] (8.245,0.09)--(8.855,1.31);
			
			\draw[-latex,line width=0.5pt] (8.245,-0.09)--(8.65,-0.9);
			\draw[line width=0.5pt] (8.245,-0.09)--(8.855,-1.31);
			
			\draw[line width=0.5pt] (11.045,1.31)--(11.655,0.09);
			\draw[-latex,line width=0.5pt] (11.045,1.31)--(11.45,0.5);
			
			\draw[-latex,line width=0.5pt] (11.045,-1.31)--(11.45,-0.5);
			\draw[line width=0.5pt] (11.045,-1.31)--(11.655,-0.09);

			\coordinate[label=above:$b_{5}$] (b5) at (-1.2,1.4);
			\draw (b5) circle(.1);
			
			\coordinate [label=below:$b_{4}$] (b4) at (-1.2,-1.4);
			\draw (b4) circle(.1);
			
			\coordinate[label=above:$b_{3}$] (b3) at (0.2,1.4);
			\draw (b3) circle(.1);
			
			\coordinate[label=below:$b_{2}$] (b2) at (0.2,-1.4);
			\draw (b2) circle(.1);
			
			\draw[line width=0.5pt] (0.1,1.4)--(-1.1,1.4);
			\draw[-latex,line width=0.5pt] (0.1,1.4)--(-0.7,1.4);
			
			\draw[line width=0.5pt] (0.1,-1.4)--(-1.1,-1.4);
			\draw[-latex,line width=0.5pt] (0.1,-1.4)--(-0.7,-1.4);
			
			\node at (-1.5,1.4) {$\cdot$};	
			\node at (-1.8,1.4) {$\cdot$};	
			\node at (-2.1,1.4) {$\cdot$};	
			\node at (-2.4,1.4) {$\cdot$};	
			\node at (-2.7,1.4) {$\cdot$};	
			\node at (-3,1.4) {$\cdot$};
				
			\coordinate [label=above:$b_{2k+1}$] (b2k+1) at (-3.3,1.4);
			\draw (b2k+1) circle(.1);
			
			\coordinate [label=below:$b_{2k}$] (b2k) at (-3.3,-1.4);
			\draw (b2k) circle(.1);
			
			\node at (-1.5,-1.4) {$\cdot$};	
			\node at (-1.8,-1.4) {$\cdot$};	
			\node at (-2.1,-1.4) {$\cdot$};	
			\node at (-2.4,-1.4) {$\cdot$};	
			\node at (-2.7,-1.4) {$\cdot$};	
			\node at (-3,-1.4) {$\cdot$};
			
			\coordinate [label=right:$b_{2k+2}$] (b2k+2) at (-4,0);
			\draw (b2k+2) circle(.1);

			\draw[line width=0.5pt] (0.155,-1.31)--(-1.155,1.31);
			\draw[-latex,line width=0.5pt] (0.155,-1.31)--(-0.3,-0.4);

			\draw[line width=0.5pt] (0.155,1.31)--(-1.155,-1.31);
			\draw[-latex,line width=0.5pt] (0.155,1.31)--(-0.3,0.4);

			\draw[line width=0.5pt] (-3.345,1.31)--(-3.955,0.09);
			\draw[-latex,line width=0.5pt] (-3.345,1.31)--(-3.75,0.5);
			
			\draw[-latex,line width=0.5pt] (-3.345,-1.31)--(-3.75,-0.5);
			\draw[line width=0.5pt] (-3.345,-1.31)--(-3.955,-0.09);
			
			\node at (-4.3,0) {$\cdot$};	
			\node at (-4.6,0) {$\cdot$};	
			\node at (-4.9,0) {$\cdot$};

			\draw (-5.2,0) circle(.1);

			\draw (-5.9,1.4) circle(.1);
			
			\draw (-5.9,-1.4) circle(.1);

			\node at (-6.2,1.4) {$\cdot$};	
			\node at (-6.5,1.4) {$\cdot$};	
			\node at (-6.8,1.4) {$\cdot$};	
			\node at (-7.1,1.4) {$\cdot$};	
			\node at (-7.4,1.4) {$\cdot$};	
			\node at (-7.7,1.4) {$\cdot$};
				
			\coordinate [label=above:$b_{(2k+1)\ell}$] (b2k+1n) at (-8,1.4);
			\draw (b2k+1n) circle(.1);
			
			\coordinate [label=below:$b_{(2k+1)\ell-1}$] (b2k+1n-1) at (-8,-1.4);
			\draw (b2k+1n-1) circle(.1);
			
			\node at (-6.2,-1.4) {$\cdot$};	
			\node at (-6.5,-1.4) {$\cdot$};	
			\node at (-6.8,-1.4) {$\cdot$};	
			\node at (-7.1,-1.4) {$\cdot$};	
			\node at (-7.4,-1.4) {$\cdot$};	
			\node at (-7.7,-1.4) {$\cdot$};
			
			\coordinate [label=left:$b_{(2k+1)\ell+1}$] (b2k+1n+1) at (-8.7,0);
			\draw (b2k+1n+1) circle(.1);

			\draw[-latex,line width=0.5pt] (-5.245,0.09)--(-5.65,0.9);
			\draw[line width=0.5pt] (-5.245,0.09)--(-5.855,1.31);
			
			\draw[-latex,line width=0.5pt] (-5.245,-0.09)--(-5.65,-0.9);
			\draw[line width=0.5pt] (-5.245,-0.09)--(-5.855,-1.31);
			
			\draw[line width=0.5pt] (-8.045,1.31)--(-8.655,0.09);
			\draw[-latex,line width=0.5pt] (-8.045,1.31)--(-8.45,0.5);
			
			\draw[-latex,line width=0.5pt] (-8.045,-1.31)--(-8.45,-0.5);
			\draw[line width=0.5pt] (-8.045,-1.31)--(-8.655,-0.09);
				
		\end{tikzpicture}
		\caption{ A broken $k$-diamond of length $2\ell$.}
		\label{fig-broken-k}
	\end{center}
\end{figure}

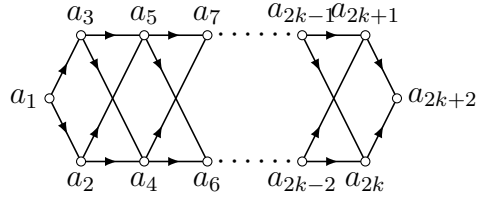
\begin{figure}[h]
	\begin{center}
		\begin{tikzpicture}[scale=0.6]	
			\draw (0,0) circle(.1);
			\node at (0,0) [above] {$a_3$};	
			\draw (1.4,0) circle(.1);
			\node at (1.4,0) [above] {$a_5$};	
			\draw (2.8,0) circle(.1);
			\node at (2.8,0) [above] {$a_7$};
			
			\node at (3.1,0) []{$\cdot$};
			\node at (3.4,0) []{$\cdot$};
			\node at (3.7,0) []{$\cdot$};
			\node at (4,0) []{$\cdot$};
			\node at (4.3,0) []{$\cdot$};
			\node at (4.6,0) []{$\cdot$};

			\draw (4.9,0) circle(.1);
			\node at (4.9,0) [above] {$a_{2k-1}$};
			\draw (6.3,0) circle(.1);
			\node at (6.3,0) [above] {$a_{2k+1}$};	
			\draw (7,-1.4) circle(.1);
			\node at (7,-1.4) [right] {$a_{2k+2}$};	
			
			\draw (-0.7,-1.4) circle(.1);
			\node at (-0.7,-1.4) [left] {$a_1$};	
			\draw (0,-2.8) circle(.1);
			\node at (0,-2.8) [below] {$a_2$};
			\draw (1.4,-2.8) circle(.1);
			\node at (1.4,-2.8) [below] {$a_4$};		
			\draw (2.8,-2.8) circle(.1);
			\node at (2.8,-2.8) [below] {$a_6$};
			
			\node at (3.1,-2.8) []{$\cdot$};
			\node at (3.4,-2.8) []{$\cdot$};
			\node at (3.7,-2.8) []{$\cdot$};
			\node at (4,-2.8) []{$\cdot$};
			\node at (4.3,-2.8) []{$\cdot$};
			\node at (4.6,-2.8) []{$\cdot$};
			
			\draw (4.9,-2.8) circle(.1);
			\node at (4.9,-2.8) [below] {$a_{2k-2}$};
			\draw (6.3,-2.8) circle(.1);
			\node at (6.3,-2.8) [below] {$a_{2k}$};
			
			\draw[-latex,line width=0.6pt](0.1,0)--(0.8,0);
			\draw[line width=0.6pt] (0.7,0)--(1.3,0);
			\draw[-latex,line width=0.6pt](1.5,0)--(2.2,0);
			\draw[line width=0.6pt] (2.1,0)--(2.7,0);
			\draw[-latex,line width=0.6pt](5,0)--(5.7,0);
			\draw[line width=0.6pt] (5.6,0)--(6.2,0);
			
			\draw[-latex,line width=0.6pt](6.345,-0.09)--(6.7,-0.8);
			\draw[line width=0.6pt] (6.6,-0.6)--(6.955,-1.31);
			
			\draw[-latex,line width=0.6pt](-0.655,-1.31)--(-0.3,-0.6);
			\draw[line width=0.6pt] (-0.4,-0.8)--(-0.045,-0.09);
			
			\draw[-latex,line width=0.6pt](-0.655,-1.49)--(-0.3,-2.2);
			\draw[line width=0.6pt] (-0.4,-2)--(-0.045,-2.71);
			
			\draw[-latex,line width=0.6pt](0.1,-2.8)--(0.8,-2.8);
			\draw[line width=0.6pt] (0.7,-2.8)--(1.3,-2.8);
			
			\draw[-latex,line width=0.6pt](1.5,-2.8)--(2.2,-2.8);
			\draw[line width=0.6pt] (2.1,-2.8)--(2.7,-2.8);
			
			\draw[-latex,line width=0.6pt](5,-2.8)--(5.7,-2.8);
			\draw[line width=0.6pt] (5.6,-2.8)--(6.2,-2.8);
			
			\draw[-latex,line width=0.6pt](6.345,-2.71)--(6.7,-2);
			\draw[line width=0.6pt] (6.6,-2.2)--(6.955,-1.49);
			
			\draw[-latex,line width=0.6pt](0.045,-0.09)--(0.4,-0.8);
			\draw[line width=0.6pt] (0.3,-0.6)--(1.355,-2.71);
			
			\draw[-latex,line width=0.6pt](1.445,-0.09)--(1.8,-0.8);
			\draw[line width=0.6pt] (1.7,-0.6)--(2.755,-2.71);
			
			\draw[-latex,line width=0.6pt](4.945,-0.09)--(5.3,-0.8);
			\draw[line width=0.6pt] (5.2,-0.6)--(6.255,-2.71);

			\draw[-latex,line width=0.6pt](0.045,-2.71)--(0.4,-2);
			\draw[line width=0.6pt] (0.3,-2.2)--(1.355,-0.09);
			
			\draw[-latex,line width=0.6pt](1.445,-2.71)--(1.8,-2);
			\draw[line width=0.6pt] (1.7,-2.2)--(2.755,-0.09);
			
			\draw[-latex,line width=0.6pt](4.945,-2.71)--(5.3,-2);
			\draw[line width=0.6pt] (5.2,-2.2)--(6.255,-0.09);
			
		\end{tikzpicture}
		
		\caption{ A $k$-elongated partition diamond of length $1$.}
		\label{fig-k-elo-1}
	\end{center}
\end{figure}
Let $\Delta_{k}(n)$ denote the number of broken $k$-diamond partitions of $n$. 
Andrews and Paule \cite{Andrews} obtained the following generating function for $\Delta_{k}(n)$,
\begin{align*}
\sum_{n=0}^{\infty} \Delta_{k}(n) q^{n}= \prod_{n=1}^{\infty} \frac{\left(1-q^{2n}\right)\left(1-q^{(2k+1)n}\right)}
{\left(1-q^n\right)^3\left(1-q^{(4k+2)n}\right)}.
\end{align*}
Note that the generating function for $\Delta_{k}(n)$ is a modular form. More precisely, $\Delta_{k}(n)$ are the coefficients of a modular function over $\Gamma_0(4k+2)$. The modular aspect led to some various arithmetic theorems and properties. For instance, many Ramanujan-like
congruences satisfied by $\Delta_{k}(n)$ has been proved by many authors, see \cite{Andrews, Chan, Hirschhorn, Hirschhorn2, Jameson, Paule, Xiong}

The Tur\'an type inequalities arise in the study of the Maclaurin coefficients of real entire functions in the Laguerre-P\'{o}lya class. We refer the interested readers to \cite{Levin} and \cite{Rahman}.

The Tur\'an type inequalities are closely related with the Jensen polynomials. The Jensen polynomial $J^{d,n}_a(X)$ of degree $d$ and shift $n$ associated to an arbitrary sequence $\{a_0, a_1, a_2, \ldots \}$ of real numbers are defined by
\begin{eqnarray*}
J^{d,n}_a(X)\colon=\sum_{j=0}^d{d\choose j}a_{n+j}X^j.
\end{eqnarray*}
More properties of the Jensen polynomials can be found in  \cite{Craven, Csordas, Csordas2}.
Note that for $d=2$ and shift $n-1$, the Jensen polynomial $J^{2,n-1}_a(X)$ is
\begin{eqnarray*}
J^{2,n-1}_a(X)=a_{n-1}+2a_{n}X+a_{n+1}X^2.
\end{eqnarray*}
$J^{2,n-1}_a(X)$ is hyperbolic if and only if the Tur\'an inequality (also called log-concavity)
\[
a_{n}^2\geq a_{n-1}a_{n+1}
\]
holds. Recall that a polynomial with real coefficients is called hyperbolic if all of its zeros are real. 

For $d=3$ and shift $n-1$, the Jensen polynomial $J^{3,n-1}_a(X)$ derives to
\begin{eqnarray*}
J^{3,n-1}_a(X)=a_{n-1}+3a_{n}X+3a_{n+1}X^2+a_{n+2}X^3.
\end{eqnarray*}
The hyperbolic of $J^{3,n-1}_a(X)$ is equivalent to the $3$-order Tur\'an inequality
\begin{equation*}
4(a_{n}^2-a_{n-1}a_{n+1})(a_{n+1}^2-a_{n}
a_{n+2})-(a_{n}a_{n+1}-a_{n-1}a_{n+2})^2\geq0.
\end{equation*}
In general, we say that the sequence $\{a_{n}\}_{n\geq 0}$ satisfies the Tur\'an inequalities of order $d$ if and only if $J^{d,n-1}_a(X)$ is hyperbolic. For the backgrounds on the Tur\'an type equalities for the Laguerre-P\'{o}lya class and the Riemann $\Xi$-function, see
\cite{Craven, Csordas, Csordas2,Dimitrov, Dimitrov2, Polya}.

Recently, 
Tur\'{a}n inequalities of order $m$ for the partition function have been proved by numerous mathematicians. Recall that a partition of a positive integer $n$ is a nonincreasing sequence $(\lambda_1,\lambda_2,\ldots, \lambda_r)$ of positive integers such that $ \lambda_1+\lambda_2+\cdots+\lambda_r=n$. Let $p(n)$ denote the number of partitions of $n$. 
  Nicolas \cite{Nicolas}, DeSalvo and Pak \cite{Desalvo} independently showed that $p(n)$ is log-concave for $n\geq 25$.  Chen \cite{Chen3} conjectured that $p(n)$ possess the Tur\'an inequality of order $3$ for $n\geq 95$, which was proved by Chen, Jia and Wang \cite{Chen-Jia-Wang-2017}. 
 Chen, Jia and Wang \cite{Chen-Jia-Wang-2017} also proposed a conjecture that for any positive integer $d\geq4$ and sufficiently large $n$, the Tur\'an inequalities of order $d$ for $p(n)$ holds. Griffin, Ono, Rolen and Zagier \cite{Zagier} verified the Jensen polynomials associated with $p(n)$ and the Riemann $\Xi$-function have only real zeros for sufficiently large $n$. Using this approach, Larson and Wagner \cite{larson} gave the thresholds for Tur\'an inequalities of order $4$ and $5$.
Griffin, Ono, Rolen, Thorner, Tripp, and Wagner \cite{Griffin} 
made this approach effective for 
the Riemann $\Xi$-function.

Chen \cite{Chen3} undertook a comprehensive study on inequalities pertaining to invariants of a binary form. In the study of the $4$-order Jensen polynomials, he considered the following three invariants of the quartic binary form which is related to the hyperbolic of the $4$-order Jensen polynomials
\begin{eqnarray}\label{invariant}
\begin{split}
A(a_0,a_1,a_2,a_3,a_4)&=a_0a_4-4a_1a_3+3a_2^2,\\[9pt]
B(a_0,a_1,a_2,a_3,a_4)&=-a_0a_2a_4+a_2^3+a_0a_3^2+a_1^2a_4-2a_1a_2a_3,\\[9pt]
I(a_0,a_1,a_2,a_3,a_4)&=A(a_0,a_1,a_2,a_3,a_4)^3-27B(a_0,a_1,a_2,a_3,a_4)^2.
\end{split}
\end{eqnarray}


In fact, $A(a_0,a_1,a_2,a_3,a_4)$ coincides with the discrete Laguerre inequalities of order $2$. The discrete Laguerre inequalities was introduced as follows \cite{wang2},
\begin{equation}\label{dislag}
L_{m}(a_n)\colon =\frac{1}{2}\sum_{k=0}^{2m}(-1)^{k+m}{{2m}\choose k}a_{n+k}a_{2m-k+n} \geq 0.
\end{equation}
Wang and Yang \cite{wang2} (\cite{wang4}, resp.) proved that the partition function, the overpartition function, the Bernoulli numbers, the derangement numbers, the Motzkin numbers, the Fine numbers, the Franel numbers and the Domb numbers (the distinct partition function, resp.)
possess the Laguerre inequality of order $2$. Wagner \cite{Wagner} proved  the partition function $p(n)$ satisfies the Laguerre inequality of
any order as $n\rightarrow \infty$ and proposed a conjecture on the thresholds of the $m$-rd Laguerre inequalities of $p(n)$ for $m\leq 10$. Dou and Wang \cite{dw} proved Wagner's conjecture for $3\leq m\leq 9$. For more work on the Laguerre inequality, refer to \cite{Cardon, Craven, Csordas5, K.Dilcher, W.H.Foster}.

On the other hand, $B(a_0,a_1,a_2,a_3,a_4)>0$ is equivalent to 
\[
\left|\begin{array}{cccc}
		a_2 & a_3 &  a_4\\
		a_1 & a_2 & a_3\\
		a_0 & a_1 & a_2\\
		\end{array} \right|>0.
\]
For $p(n)$, $B>0$ has been showed by Hou and Zhang \cite{hz}, and Jia and Wang \cite{Jia-Wang-2018} respectively. Recently, Wang and Yang \cite{wang3} proved the positivity of $\det (a_{n-i+j})_{1\leq i,j\leq 4}$ for sequence $\{a_n\}_{n\geq0}$ involving the partition function and the overpartition function and gave an iterated approach to compute $\det (a_{n-i+j})_{1\leq i,j\leq m}$ for any $m$.

Since then, Tur\'{a}n inequalities for other partition functions have been extensively investigated. For example, Craig and Pun \cite{Craig} showed that $q(n)$ satisfies the Tur\'an inequalities of order $d$ for sufficiently large $n$ and conjectured that the distinct partition function $q(n)$ is log-concave for $n\geq 33$ and satisfies the Tur\'an inequalities of order $3$ for $n\geq 121$, which were proved by Dong and Ji \cite{DJ}. Recall that a distinct partition of integer $n$ is a partition of $n$ with all distinct parts. Dong and  Ji \cite{DJ} also proposed a conjecture that for $q(n)$, $A> 0$ for $n\geq 230$, $B> 0$ for $n\geq 272$ and $I> 0$ for $n\geq 267$. Wang and Yang \cite{wang4} affirmed that for $q(n)$, $A> 0$ for $n\geq 229$, $B> 0$ for $n\geq 271$. Dong, Ji and Jia \cite{Dong} proved that $\Delta_{k}(n)$ satisfies the order $d$ Tur\'{a}n inequalities for $d\geq 1$ and for sufficiently large $n$ when $k = 1$ or $2$. Jia \cite{Jia} proved that $\{\Delta_{k}(n)\}_{n\geq6}$ satisfies $3$-order Tur\'{a}n inequalities for $k = 1$ or $2$.

Similar to the work of Wang and Yang \cite{wang4} with $q(n)$ just mentioned, we will derive that $\Delta_{k}(n)$ satisfies $A> 0$ for $n\geq 12$, $B> 0$ for $n\geq 18$ when $k = 1$ or $2$ and conjecture that $I>0$ for $n\geq 14$ in this paper. Our main tool is due to Dou and Wang \cite{dw},  Wang and Yang \cite{wang3} and Wang and Yang \cite{wang4}. The remaining of this paper is organized as follows. In Section \ref{1}, we will provide some auxiliary results required to show the main theorems of this paper. In Section \ref{2}, using the spirit in \cite{dw}, we will prove some inequalities involving the $2$-nd modified Bessel function of the first kind, which are necessary for proofing our main Theorems. In Section \ref{3}, we will prove $\Delta_{k}(n)$ satisfies the Laguerre inequality of order $2$ for $n\geq 12$, i.e. $A>0$. In Section \ref{4}, using a similar approach as in Section \ref{3}, we will show the positivity of $\det (\Delta_{k}(n-2+i+j))_{1\leq i,j\leq 3}$ for $n\geq 18$, i.e. $B>0$. We conclude in Section \ref{5} with some problems for further work.

\section{Preparation}\label{1}
In this Section, we will provide some auxiliary results essential to prove the main theorems of this paper. The proofs involve some sharper bounds for $\Delta_{k}(n)$ than those in \cite{Dong}. Recall that Dong, Ji and Jia \cite{Dong} gave the following bounds for $\Delta_{k}(n)$.

\begin{theo}[Dong, Ji and Jia \cite{Dong}, Theorem 3.1]\label{Dong, Ji and Jia}
	
  For $k = 1$ or $2$, let $x_{k}(n)$ be
	\begin{eqnarray}
	x_{k}(n)=\frac{\pi\sqrt{24n-(2k+2)}}{6},
	\end{eqnarray}
and 
\begin{eqnarray}
	\alpha_{k}=\frac{5k+2}{2k+1}.
	\end{eqnarray}
Then for $k = 1$ or $2$ and for $x_{k}(n)\geq 152$, or equivalently, for $n\geq 3512$, we have
\begin{eqnarray}\label{Ineq q(n)}
	M_{k}(n)\left(1-\frac{1}{x_{k}(n)^{6}}\right)\leq \Delta_{k}(n)\leq M_{k}(n)\left(1+\frac{1}{x_{k}(n)^{6}}\right),
	\end{eqnarray}
where
\begin{equation}
M_{k}(n)=\frac{\alpha_{k}\pi^3}{18x_{k}(n)^{2}}I_{2}\left(\sqrt{\alpha_{k}}x_{k}(n)\right),
    \end{equation}
and $I_2$ is the $2$-nd modified Bessel function of the first kind defined as   
   \begin{eqnarray}\label{I_2(s)}
			I_2(s)=\frac{s^2}{3\pi}\int_{-1}^{1}{(1-t^{2})}^{\frac{3}{2}}e^{st}dt.
	\end{eqnarray}

\end{theo}

We discover that Theorem \ref{Dong, Ji and Jia} is precise to prove our Theorem \ref{theorem1} but not precise enough to prove our Theorem \ref{3jie-theorem}. Thus, to achieve our goals,  we will make a slight adjustment to the proof of Theorem \ref{Dong, Ji and Jia} and make it effective for the following theroem.
\begin{theo}\label{generalize}
For $k = 1$ or $2$ and for $x_{k}(n)\geq 152$, or equivalently, for $n\geq 3512$, we have
\begin{eqnarray}\label{Generalize}
	M_{k}(n)\left(1-\frac{1}{x_{k}(n)^{10}}\right)\leq \Delta_{k}(n)\leq M_{k}(n)\left(1+\frac{1}{x_{k}(n)^{10}}\right).
	\end{eqnarray}
\end{theo}

The majority of its proof is similar to the proof of Theorem \ref{Dong, Ji and Jia}, the only essential modification is that we convert the inequality \cite[(3.8)]{DJ} \[72exp\left(-\frac{\sqrt{21}}{6}x_k(n)\right)\leq\frac{1}{2{x_k}^6(n)} \quad \text{for} \quad x_{k}(n)\geq 152\] to \[72exp\left(-\frac{\sqrt{21}}{6}x_k(n)\right)\leq\frac{1}{2{x_k}^{10}(n)} \quad \text{for} \quad x_{k}(n)\geq 152.\] Thus, we omit the details.

In order to make the asymptotic formula in Theorem \ref{Dong, Ji and Jia} (Theorem \ref{generalize}, resp.) useful in the proof of
Theorem \ref{theorem1} (Theorem \ref{3jie-theorem}, resp.), we need inequality \eqref{Ineq I_2(s)} (\eqref{ineq I_2(s)}, resp.) on the $2$-nd modified Bessel function $I_2(s)$
of the first kind as follows,
\begin{lemm}\label{lemm 1}
Let
  \begin{eqnarray}
	B_I(s)\colon =1-\frac{15}{8s}+\frac{105}{128s^2}+\frac{315}{1024s^3}+\frac{10395}{32768s^4}+\frac{135135}{262144s^5},
	\end{eqnarray}
then for $s\geq 28$, we have
\begin{eqnarray}\label{Ineq I_2(s)}
	\frac{e^{s}}{\sqrt{2\pi s}}\left(B_I(s)-\frac{27}{s^6}\right)\leq I_2(s)\leq \frac{e^{s}}{\sqrt{2\pi s}}\left(B_I(s)+\frac{27}{s^6}\right).
	\end{eqnarray}
\end{lemm}

For convenience, let's denote the lower (upper, resp.) bound for $I_2(s)$ as $I_{21}(s)$ ($I_{22}(s)$, resp.),
\begin{eqnarray}\label{12}
			I_{21}(s)\colon =\frac{e^{s}}{\sqrt{2\pi s}}\left(B_I(s)-\frac{27}{s^6}\right),\quad
            I_{22}(s)\colon =\frac{e^{s}}{\sqrt{2\pi s}}\left(B_I(s)+\frac{27}{s^6}\right).
	\end{eqnarray}

We find that Jia \cite[Theorem 2.1]{Jia} already derived that $\left\lvert I_2(s)-B_I(s) \right\rvert \leq \frac{73}{s^6}$ for $s\geq \frac{\left(\frac{15}{2}\right)^6}{120}\approx1483.2$, but we do not apply this bounds because the needed range of $s$ is too large to facilitate computer verification of our Theorem \ref{theorem1}.

\begin{lemm}\label{B_I}
Let
  \begin{eqnarray*}
  \begin{split}
\bar{B}_I(s)\colon &=1-\frac{15}{8s}+\frac{105}{128s^{2}}+\frac{315}{1024s^{3}}+\frac{10395}{32768s^{4}}\\[9pt]
&+\frac{135135}{262144s^{5}}+\frac{4729725}{4194304s^{6}}+\frac{103378275}{33554432s^{7}}+\frac{21606059475}{2147483648s^{8}}\\[9pt]
&+\frac{655383804075}{17179869184s^{9}}+\frac{45221482481175}{274877906944s^{10}}+\frac{1747193641318125}{2199023255552s^{11}}.
	 \end{split}
\end{eqnarray*}
Then for $s\geq 50$, 
\begin{eqnarray}\label{ineq I_2(s)}
	\frac{e^{s}}{\sqrt{2\pi s}}\left(\bar{B}_I(s)-\frac{6148836}{s^{12}}\right)\leq I_2(s)\leq \frac{e^{s}}{\sqrt{2\pi s}}\left(\bar{B}_I(s)+\frac{6148836}{s^{12}}\right).
	\end{eqnarray}
\end{lemm}

For convenience, 
let's denote the lower (upper, resp.) bound for $I_2(s)$ as $I_{23}(s)$ ($I_{24}(s)$, resp.),
\begin{eqnarray}\label{12}
  \begin{split}
			I_{23}(s)\colon &=\frac{e^{s}}{\sqrt{2\pi s}}\left(\bar{B}_I(s)-\frac{6148836}{s^{12}}\right),\\[9pt]
            I_{24}(s)\colon &=\frac{e^{s}}{\sqrt{2\pi s}}\left(\bar{B}_I(s)+\frac{6148836}{s^{12}}\right).
   \end{split}
	\end{eqnarray}

 Inspired by the method to prove the explicit bounds for the first modified Bessel function $I_1(s)$ used by Dong and Ji \cite[Section 2]{DJ}, we can similarly obtain the explicit bounds for the $2$-nd modified Bessel function $I_2(s)$ as shown in Lemma \ref{lemm 1} and Lemma \ref{B_I}. Due to the similarities between the two lemmas' methods of proving, we will only offer proof for Lemma \ref{lemm 1} in next Section.
\section{A proof for Lemma \ref{lemm 1}}\label{2}
Before doing this, let us first recall the definitions of the Gamma function $\Gamma(a)$ and the upper incomplete Gamma function $\Gamma(a,s)$, see \cite[Chapter6]{Abramowitz}

The Gamma function $\Gamma(a)$ is defined by
\begin{eqnarray}\label{Gamma}
			\Gamma(a)=\int_{0}^{\infty}t^{a-1}e^{-t}dt.
	\end{eqnarray}

The upper incomplete Gamma function $\Gamma(a,s)$ is defined by
\begin{eqnarray}\label{Gammas}
			\Gamma(a,s)=\int_{s}^{\infty}t^{a-1}e^{-t}dt.
	\end{eqnarray}
The following estimate on $\Gamma(a,s)$ can be derived from the proof of Proposition 2.6 of Pinelis \cite{Pinelis} which is required for proving Lemma \ref{lemm 1}. For $a\geq1$ and $s\geq a$,
\begin{eqnarray}\label{G,s}
			\Gamma(a,s)\leq as^{a-1}e^{-s}.
	\end{eqnarray}
Next we will prove the Lemma \ref{lemm 1}.

\begin{proof}
We start with the integral definition \eqref{I_2(s)} of $I_2(s)$ ,
\begin{eqnarray}\label{I_2}
			I_2(s)=\frac{s^2}{3\pi}\int_{0}^{1}{(1-t^{2})}^{\frac{3}{2}}e^{st}dt+\frac{s^2}{3\pi}\int_{-1}^{0}{(1-t^{2})}^{\frac{3}{2}}e^{st}dt.
	\end{eqnarray}
It is clear that
\begin{eqnarray}\label{I_20}
			\left\lvert\frac{s^2}{3\pi}\int_{-1}^{0}{(1-t^{2})}^{\frac{3}{2}}e^{st}dt\right\rvert\leq\frac{s^2}{3\pi}.
	\end{eqnarray}
We next estimate the first integral in \eqref{I_2}. Setting $u=1-t$, we get that
\begin{eqnarray}\label{I_21}
			\frac{s^2}{3\pi}\int_{0}^{1}{(1-t^{2})}^{\frac{3}{2}}e^{st}dt=\frac{s^2e^s}{3\pi}\int_{0}^{1}{(2-u)}^{\frac{3}{2}}{u}^{\frac{3}{2}}e^{-su}du.
	\end{eqnarray}
By Taylor’s Theorem, we have
\begin{eqnarray}\label{I_22}
			{(2-u)}^{\frac{3}{2}}=2\sqrt2-\frac{3u}{\sqrt2}+\frac{3u^2}{8\sqrt2}+\frac{u^3}{32\sqrt2}+\frac{3u^4}{512\sqrt2}+\frac{3u^5}{2048\sqrt2}+c(\xi)u^6.
	\end{eqnarray}
where for some $\xi\in[0,1]$,
\begin{eqnarray}\label{I_23}
			c(\xi)=\frac{1}{6!}\left(\frac{d^6}{du^6}(2-u)^{\frac{3}{2}}\right)_{u=\xi}=\frac{7}{1024}(2-\xi)^{-\frac{9}{2}}.
	\end{eqnarray}
Substituting \eqref{I_22} into \eqref{I_21}, we obtain
\begin{small}
\begin{align}\label{I_24}		\frac{s^2e^s}{3\pi}&\int_{0}^{1}\left(2\sqrt2{u}^{\frac{3}{2}}-\frac{3{u}^{\frac{5}{2}}}{\sqrt2}+\frac{3{u}^{\frac{7}{2}}}{8\sqrt2}+\frac{{u}^{\frac{9}{2}}}{32\sqrt2}
+\frac{3{u}^{\frac{11}{2}}}{512\sqrt2}+\frac{3{u}^{\frac{13}{2}}}{2048\sqrt2}+c(\xi){u}^{\frac{15}{2}}\right)e^{-su}du\nonumber \\[9pt]
&=\frac{\sqrt2s^2e^s}{3\pi}\left(\int_{0}^{\infty}-\int_{1}^{\infty}\right)\left(2{u}^{\frac{3}{2}}-\frac{3{u}^{\frac{5}{2}}}{2}+\frac{3{u}^{\frac{7}{2}}}{16}+\frac{{u}^{\frac{9}{2}}}{64}
+\frac{3{u}^{\frac{11}{2}}}{1024}+\frac{3{u}^{\frac{13}{2}}}{4096}\right)e^{-su}du\notag \\[9pt]
&~~+\frac{s^2e^s}{3\pi}\int_{0}^{1}c(\xi){u}^{\frac{15}{2}}e^{-su}du:=I_2^{(1)}(s)+I_2^{(2)}(s)+I_2^{(3)}(s).
\end{align}
\end{small}
Evaluating the first integral $I_2^{(1)}(s)$ in \eqref{I_24} yields the main term,
\begin{small}
\begin{align}\label{I_25}			I&_2^{(1)}(s)=\frac{\sqrt2s^2e^s}{3\pi}\int_{0}^{\infty}\left(2{u}^{\frac{3}{2}}-\frac{3{u}^{\frac{5}{2}}}{2}+\frac{3{u}^{\frac{7}{2}}}{16}+\frac{{u}^{\frac{9}{2}}}{64}
+\frac{3{u}^{\frac{11}{2}}}{1024}+\frac{3{u}^{\frac{13}{2}}}{4096}\right)e^{-su}du\notag \\[9pt]
=&\frac{\sqrt2e^s}{3\sqrt s\pi}\int_{0}^{\infty}\left(2{(su)}^{\frac{3}{2}}-\frac{3{(su)}^{\frac{5}{2}}}{2s}+\frac{3{(su)}^{\frac{7}{2}}}{16s^2}+\frac{{(su)}^{\frac{9}{2}}}{64s^3}
+\frac{3{(su)}^{\frac{11}{2}}}{1024s^4}+\frac{3{(su)}^{\frac{13}{2}}}{4096s^5}\right)e^{-su}d(su)\notag \\[9pt]
&\overset{\eqref{Gamma}}{=}\frac{\sqrt2e^s}{3\sqrt s\pi}\left(2\Gamma\left(\frac{5}{2}\right)-\frac{3}{2s}\Gamma\left(\frac{7}{2}\right)+\frac{3}{16s^2}\Gamma\left(\frac{9}{2}\right)+\frac{1}{64s^3}\Gamma\left(\frac{11}{2}\right)
\notag\right.\\ &~~~~~~~~~~~~~~~~~~ \left.
+\frac{3}{1024s^4}\Gamma\left(\frac{13}{2}\right)+\frac{3}{4096s^5}\Gamma\left(\frac{15}{2}\right)\right)\notag\\[9pt]
&=\frac{e^s}{\sqrt{2\pi s}}\left(1-\frac{15}{8s}+\frac{105}{128s^2}+\frac{315}{1024s^3}
+\frac{10395}{32768s^4}+\frac{135135}{262144s^5}\right). 
\end{align}
\end{small}
We proceed to evaluate the second integral $I_2^{(2)}(s)$ in \eqref{I_24},
\begin{small}
\begin{align}\label{I_26}			
\left\lvert I_2^{(2)}(s)\right\rvert&=\left\lvert-\frac{\sqrt2s^2e^s}{3\pi}\int_{1}^{\infty}\left(2{u}^{\frac{3}{2}}-\frac{3{u}^{\frac{5}{2}}}{2}+\frac{3{u}^{\frac{7}{2}}}{16}+\frac{{u}^{\frac{9}{2}}}{64}
+\frac{3{u}^{\frac{11}{2}}}{1024}+\frac{3{u}^{\frac{13}{2}}}{4096}\right)e^{-su}du\right\rvert\notag \\[9pt]
\leq\frac{\sqrt2e^s}{3\sqrt s\pi}&\int_{s}^{\infty}\left(2{(su)}^{\frac{3}{2}}+\frac{3{(su)}^{\frac{5}{2}}}{2s}+\frac{3{(su)}^{\frac{7}{2}}}{16s^2}+\frac{{(su)}^{\frac{9}{2}}}{64s^3}
+\frac{3{(su)}^{\frac{11}{2}}}{1024s^4}+\frac{3{(su)}^{\frac{13}{2}}}{4096s^5}\right)e^{-su}d(su)\notag \\[9pt]
\overset{\eqref{Gammas}}{=}&\frac{\sqrt2e^s}{3\sqrt s\pi}\left(2\Gamma\left(\frac{5}{2},s\right)+\frac{3}{2s}\Gamma\left(\frac{7}{2},s\right)+\frac{3}{16s^2}\Gamma\left(\frac{9}{2},s\right)+\frac{1}{64s^3}\Gamma\left(\frac{11}{2},s\right)
\notag\right.\\ &~~~~~~~~~~~~~~~~~~ \left.
+\frac{3}{1024s^4}\Gamma\left(\frac{13}{2},s\right)+\frac{3}{4096s^5}\Gamma\left(\frac{15}{2},s\right)\right)\notag\\[9pt]
\overset{\eqref{G,s}}{\leq}&\frac{\sqrt2e^s}{3\sqrt s\pi}\left(2\cdot\frac{5}{2}+\frac{3}{2}\cdot\frac{7}{2}+\frac{3}{16}\cdot\frac{9}{2}
+\frac{1}{64}\cdot\frac{11}{2}+\frac{3}{1024}\cdot\frac{13}{2}+\frac{3}{4096}\cdot\frac{15}{2}\right)s^{\frac{3}{2}}e^{-s}\notag\\[9pt]
=&\frac{30595s}{4096\sqrt2\pi}\quad \text{for} \quad s\geq\frac{15}{2}. 
\end{align}
\end{small}
It remains to estimate $I_2^{(3)}(s)$ in \eqref{I_24}. From \eqref{I_23}, we see that
\begin{align}\label{I_27}
			\left\lvert I_2^{(3)}(s)\right\rvert=\left\lvert\frac{s^2e^s}{3\pi}\int_{0}^{1}c(\xi){u}^{\frac{15}{2}}e^{-su}du\right\rvert
&\leq\frac{7s^{-\frac{13}{2}}e^s}{3072\pi}\int_{0}^{\infty}{(su)}^{\frac{15}{2}}e^{-su}d(su)\notag\\[9pt]
&=\frac{7s^{-\frac{13}{2}}e^s}{3072\pi}\Gamma\left(\frac{17}{2}\right)\notag\\[9pt]
&=\frac{4729725}{262144\sqrt\pi}s^{-\frac{13}{2}}e^s.
	\end{align}
Combining \eqref{I_2}, \eqref{I_20}, \eqref{I_24}, \eqref{I_25}, \eqref{I_26} and \eqref{I_27}, we derive that for $s\geq\frac{15}{2}$,
\begin{equation}\label{I_28}
\resizebox{0.95\hsize}{!}{$\begin{aligned}
			I_2(s)=\frac{e^s}{\sqrt{2\pi s}}\left(1-\frac{15}{8s}+\frac{105}{128s^2}+\frac{315}{1024s^3}
+\frac{310395}{32768s^4}+\frac{135135}{262144s^5}\right)+f(s),
\end{aligned}$}
	\end{equation}
where
\begin{align}\label{I_29}
			\left\lvert f(s)\right\rvert&\leq\frac{s^2}{3\pi}+\frac{30595s}{4096\sqrt2\pi}+\frac{4729725}{262144\sqrt\pi}s^{-\frac{13}{2}}e^s\notag\\[9pt]
&=\frac{e^s}{\sqrt{2\pi s}}\cdot\frac{1}{s^6}\left(\frac{91785+4096\sqrt2s}{12288\sqrt{\pi}}e^{-s}s^{\frac{15}{2}}+\frac{4729725}{131072\sqrt2}\right)\\[9pt]
\colon&=\frac{e^s}{\sqrt{2\pi s}}\cdot\frac{1}{s^6}\cdot g(s)\notag.
	\end{align}
Note that the derivative of $g(s)$ with respect to $s$ is 
\begin{equation*}
			g^{'}(s)=-\frac{e^{-s}s^{\frac{13}{2}}}{24576\sqrt{\pi}}\left(8192\sqrt2s^2+\left(183570-69632\sqrt2\right)s-1376775\right).
	\end{equation*}
Since $g^{'}(s)<0$ for $s \geq 7.9$, we deduce that $g(s)$ is decreasing when $s \geq 7.9$. This implies
that when $s\geq 28$,
\begin{equation*}
			g(s)\leq g(28)\approx26.09159<27.
	\end{equation*}
Hence 
\begin{equation}\label{I_210}
			\left\lvert f(s)\right\rvert\leq\frac{e^s}{\sqrt{2\pi s}}\cdot\frac{27}{s^6}.
	\end{equation} Combining \eqref{I_28} and \eqref{I_210}, we are led to \eqref{Ineq I_2(s)}
in Lemma \ref{lemm 1}. This completes the proof.\qed
\end{proof}

\section{A proof for Theorem \ref{theorem1}}\label{3}
In this section, we will prove the broken $k$-diamond partition function $\Delta_{k}(n)$ satisfies the Laguerre inequalities of order $2$ with the aid of Theorem \ref{Dong, Ji and Jia} and Lemma \ref{lemm 1}.

\begin{theo}\label{theorem1}
Let $\Delta_{k}(n)$ denote the broken $k$-diamond partition function. For $n\geq 12$ and $k=1$ or $2$, we have
	\begin{eqnarray}\label{Lag Ineq q(n)}
			3\Delta_{k}(n+2)^2-4\Delta_{k}(n+1)\Delta_{k}(n+3)+\Delta_{k}(n)\Delta_{k}(n+4)>0.
	\end{eqnarray}
\end{theo}

\begin{proof}
Note that for $n\geq 3512$, \eqref{Ineq q(n)} can be rewritten as
	\begin{eqnarray}\label{Ineq q(n)1}
\resizebox{0.93\hsize}{!}{$\begin{aligned}
	\frac{\alpha_{k}\pi^{3}}{18x_{k}(n)^{8}}I_2\left(\sqrt{\alpha_{k}}x_{k}(n)\right)\left({x_{k}(n)}^6-1\right)\leq \Delta_{k} \leq \frac{\alpha_{k}\pi^{3}}{18x_{k}(n)^{8}}I_2\left(\sqrt{\alpha_{k}}x_{k}(n)\right)\left({x_{k}(n)}^6+1\right),
\end{aligned}$}
	\end{eqnarray}

Let
\begin{eqnarray}\label{u}
	u_{k}(n)\colon =\sqrt{x_{k}(n)}=\left[\frac{\pi^{2}(24n-(2k+2))}{36}\right]^{\frac{1}{4}},
     \end{eqnarray}
then the inequality \eqref{Ineq q(n)1} 
will be converted into
 \begin{eqnarray}\label{Ineq q(n)2}
	\frac{\alpha_{k}\pi^{3}}{18u^{16}}I_2\left(\sqrt{\alpha_{k}}u^2\right)\left({u}^{12}-1\right)\leq \Delta_{k} \leq \frac{\alpha_{k}\pi^{3}}{18u^{16}}I_2\left(\sqrt{\alpha_{k}}u^2\right)\left({u}^{12}+1\right),
	\end{eqnarray}
where $u$ is the abbreviation for $u_{k}(n)$.

  For convenience, we denote
\begin{eqnarray}\label{beta}
	\beta(t)=t^{12}+1,~~\gamma(t)=t^{12}-1,
	\end{eqnarray}
 and
   \begin{eqnarray}
	f_{k}(n)\colon =\frac{\alpha_{k}\pi^{3}}{18u^{16}}\gamma(u)I_{21}\left(\sqrt{\alpha_{k}}u^2\right),\\[9pt]
    g_{k}(n)\colon =\frac{\alpha_{k}\pi^{3}}{18u^{16}}\beta(u)I_{22}\left(\sqrt{\alpha_{k}}u^2\right).
     \end{eqnarray}
      Since $I_{21}(s)\leq I_2(s) \leq I_{22}(s)$ holds for $s\geq 28$ as shown in Lemma \ref{lemm 1}, then we have $f_{k}(n) \leq \Delta_{k}(n) \leq g_{k}(n)$ for $n\geq 3512$.
    Thus, in order to prove Theorem \ref{theorem1}, it is sufficient to show that 
 \begin{equation*}
3f_{k}(n+2)^2-4g_{k}(n+1)g_{k}(n+3)+f_{k}(n)f_{k}(n+4)>0,
 \end{equation*}
which can be rewritten as
  \begin{equation}\label{key}
3-\frac{4g_{k}(n+1)g_{k}(n+3)}{f_{k}(n+2)^2}+\frac{f_{k}(n)f_{k}(n+4)}{f_{k}(n+2)^2}>0.
 \end{equation}
 To prove it, we denote
\begin{eqnarray}\label{Eq u}
\resizebox{0.935\hsize}{!}{$\begin{aligned}
	w=u_{k}(n),~y=u_{k}(n+1),~z=u_{k}(n+2),~i=u_{k}(n+3),~j=u_{k}(n+4),
\end{aligned}$}
	\end{eqnarray}
then with the aid of Mathematica, the left-hand side of \eqref{key} can be simplified to
\begin{eqnarray}\label{Eq key}
	\frac{h_1e^{\sqrt{\alpha_{k}}\left(i^2+y^2-2z^2\right)}+h_2e^{\sqrt{\alpha_{k}}\left(j^2+w^2-2z^2\right)}+h_3}{h_3},
	\end{eqnarray}
where
\begin{eqnarray}\label{Eq h_i}
   \begin{split}
	&h_1=F_{k}(n)j^{28}w^{28}z^{58}\beta(i)\beta(y),&~~~~~~~
    h_2=G_{k}(n)i^{28}y^{28}z^{58}\gamma(j)\gamma(w),\\[9pt]
    &h_3=H_{k}(n)i^{28}j^{28}w^{28}y^{28}\gamma(z)^2,&
     \end{split}
\end{eqnarray}
and 
\begin{eqnarray}\label{Eq f,g,h}
  \begin{split}
	&F_{k}(n)=\sum_{m,n=0}^{6}a_{m,n}(k)jwi^{2m} y^{2n},&~~~~~~~~~~
   G_{k}(n)=\sum_{m,n=0}^{6}b_{m,n}(k)iy j^{2m} w^{2n},\\[9pt]
    &H_{k}(n)=\sum_{t=0}^{12}c_{t}(k)ijwyz^{2t},&
    \end{split}
\end{eqnarray}
where $a_{m,n}(k)$, $b_{m,n}(k)$ and $c_{t}(k)$ are related to $k$ and are the real coefficients of polynomials $F_{k}(n)$, $G_{k}(n)$ and $H_{k}(n)$ which can be computed by Mathematica, respectively. 

Now we proceed to prove \eqref{Eq key} is positive for $n\geq3512$. Applying \eqref{u} and \eqref{Eq u} into the expression of  $H_{k}(n)$, one can easily deduce that $H_{k}(n)>0$ for all $n\geq 1$ by Mathematica, which implies that the denominator $h_3$ of \eqref{Eq key} is positive for $n\geq 1$. 
Thus, all that is required of us is to show
\begin{eqnarray}\label{key Eq}
	h_1e^{\sqrt{\alpha_{k}}\left(i^2+y^2-2z^2\right)}+h_2e^{\sqrt{\alpha_{k}}\left(j^2+w^2-2z^2\right)}+h_3>0.
	\end{eqnarray}

In order to accomplish it, we need to estimate $h_1$, $h_2$, $h_3$, $e^{\sqrt{\alpha_{k}}\left(i^2+y^2-2z^2\right)}$ and $e^{\sqrt{\alpha_{k}}\left(j^2+w^2-2z^2\right)}$. We prefer to give the estimates of $w$, $y$, $i$ and $j$ via the following equalities. For $n\geq 1$
\begin{eqnarray}\label{Eq w,y,i,j}
	w=\sqrt[4]{z^4-\frac{4\pi^{2}}{3}},~y=\sqrt[4]{z^4-\frac{2\pi^{2}}{3}},~i=\sqrt[4]{z^4+\frac{2\pi^{2}}{3}},~j=\sqrt[4]{z^4+\frac{4\pi^{2}}{3}}.
\end{eqnarray}
We can obtain the following expansions
\begin{align*}
 w&=z-\frac{\pi^{2}}{3z^{3}}-\frac{\pi^{4}}{6z^{7}}-\frac{7\pi^{6}}{54z^{11}}-\frac{77\pi^{8}}{648z^{15}}-\frac{77\pi^{10}}{648z^{19}}+O\left(\frac{1}{z^{23}}\right),\\[9pt]
 y&=z-\frac{\pi^{2}}{6z^{3}}-\frac{\pi^{4}}{24z^{7}}-\frac{7\pi^{6}}{432z^{11}}-\frac{77\pi^{8}}{10368z^{15}}-\frac{77\pi^{10}}{20736z^{19}}+O\left(\frac{1}{z^{23}}\right),\\[9pt]
 i&=z+\frac{\pi^{2}}{6z^{3}}-\frac{\pi^{4}}{24z^{7}}+\frac{7\pi^{6}}{432z^{11}}-\frac{77\pi^{8}}{10368z^{15}}+\frac{77\pi^{10}}{20736z^{19}}+O\left(\frac{1}{z^{23}}\right),\\[9pt]
 j&=z+\frac{\pi^{2}}{3z^{3}}-\frac{\pi^{4}}{6z^{7}}+\frac{7\pi^{6}}{54z^{11}}-\frac{77\pi^{8}}{648z^{15}}+\frac{77\pi^{10}}{648z^{19}}+O\left(\frac{1}{z^{23}}\right).
\end{align*}
   By the expansions of $w$, $y$, $i$ and $j$, we have that for $n\geq 122$,
\begin{eqnarray}\label{Ineq w,y,i,j}
 \begin{split}
	w_1&<w<w_2,&~~~~~~
    y_1&<y<y_2,\\[9pt]
    i_1&<i<i_2,&
    j_1&<j<j_2,
 \end{split}
\end{eqnarray}
where
\begin{align} \label{w_j}
    w_1&=z-\frac{\pi^{2}}{3z^{3}}-\frac{\pi^{4}}{6z^{7}}-\frac{7\pi^{6}}{54z^{11}}-\frac{77\pi^{8}}{648z^{15}}-\frac{39\pi^{10}}{324z^{19}},\notag\\[9pt]
    w_2&=z-\frac{\pi^{2}}{3z^{3}}-\frac{\pi^{4}}{6z^{7}}-\frac{7\pi^{6}}{54z^{11}}-\frac{77\pi^{8}}{648z^{15}}-\frac{77\pi^{10}}{648z^{19}},\\[9pt]
    y_1&=z-\frac{\pi^{2}}{6z^{3}}-\frac{\pi^{4}}{24z^{7}}-\frac{7\pi^{6}}{432z^{11}}-\frac{77\pi^{8}}{10368z^{15}}-\frac{13\pi^{10}}{3456z^{19}},\notag\\[9pt]
    y_2&=z-\frac{\pi^{2}}{6z^{3}}-\frac{\pi^{4}}{24z^{7}}-\frac{7\pi^{6}}{432z^{11}}-\frac{77\pi^{8}}{10368z^{15}}-\frac{77\pi^{10}}{20736z^{19}},\notag\\[9pt]
    i_1&=z+\frac{\pi^{2}}{6z^{3}}-\frac{\pi^{4}}{24z^{7}}+\frac{7\pi^{6}}{432z^{11}}-\frac{77\pi^{8}}{10368z^{15}},\notag\\[9pt]
    i_2&=z+\frac{\pi^{2}}{6z^{3}}-\frac{\pi^{4}}{24z^{7}}+\frac{7\pi^{6}}{432z^{11}}-\frac{77\pi^{8}}{10368z^{15}}+\frac{77\pi^{10}}{20736z^{19}},\notag\\[9pt]
    j_1&=z+\frac{\pi^{2}}{3z^{3}}-\frac{\pi^{4}}{6z^{7}}+\frac{7\pi^{6}}{54z^{11}}-\frac{77\pi^{8}}{648z^{15}},\notag\\[9pt]
    j_2&=z+\frac{\pi^{2}}{3z^{3}}-\frac{\pi^{4}}{6z^{7}}+\frac{7\pi^{6}}{54z^{11}}-\frac{77\pi^{8}}{648z^{15}}+\frac{77\pi^{10}}{648z^{19}}.\notag
\end{align}

Next we turn to estimate $h_1$, $h_2$ and $h_3$, which requires the estimation of $F_{k}(n)$, $G_{k}(n)$ and $H_{k}(n)$. As a result, we must estimate each term in each summation expressions.

Let\begin{eqnarray*}
\resizebox{0.9\hsize}{!}{$\begin{aligned}
I_1^{2m}=\left\{ \begin{array}{ll}
i^{2m} &\textrm {$ m\equiv 0\left(mod~2\right)$}\\
i^{2m-2}i_1^2&\textrm {$ m\equiv 1\left(mod~2\right)$}\\
\end{array}\right.,
I_2^{2m}=\left\{ \begin{array}{ll}
i^{2m}&\textrm {$m\equiv 0\left(mod~2\right)$}\\
i^{2m-2}i_2^2&\textrm {$m\equiv 1\left(mod~2\right)$}\\
\end{array}\right.,
\end{aligned}$}
\end{eqnarray*}

\begin{eqnarray}\label{J}
\resizebox{0.9\hsize}{!}{$\begin{aligned}
J_1^{2m}=\left\{ \begin{array}{ll}
j^{2m} &\textrm {$ m\equiv 0\left(mod~2\right)$}\\
j^{2m-2}j_1^2 &\textrm {$ m\equiv 1\left(mod~2\right)$}\\
\end{array}\right.,
J_2^{2m}=\left\{ \begin{array}{ll}
j^{2m}&\textrm {$m\equiv 0\left(mod~2\right)$}\\
j^{2m-2}j_2^2&\textrm {$m\equiv 1\left(mod~2\right)$}\\
\end{array}\right.,
\end{aligned}$}
\end{eqnarray}

\begin{eqnarray*}
\resizebox{0.9\hsize}{!}{$\begin{aligned}
W_1^{2n}=\left\{ \begin{array}{ll}
w^{2n} &\textrm {$ n\equiv 0\left(mod~2\right)$}\\
w^{2n-2}w_1^2 &\textrm {$ n\equiv 1\left(mod~2\right))$}\\
\end{array}\right.,
W_2^{2n}=\left\{ \begin{array}{ll}
w^{2n}&\textrm {$n\equiv 0\left(mod~2\right)$}\\
w^{2n-2}w_2^2&\textrm {$n\equiv 1\left(mod~2\right)$}\\
\end{array}\right.,
\end{aligned}$}
\end{eqnarray*}

\begin{eqnarray*}\label{Y}
\resizebox{0.9\hsize}{!}{$\begin{aligned}
Y_1^{2n}=\left\{ \begin{array}{ll}
y^{2n}&\textrm {$n\equiv 0\left(mod~2\right)$}\\
y^{2n-2}y_1^2&\textrm {$n\equiv 1\left(mod~2\right)$}\\
\end{array}\right.,
Y_1^{2n}=\left\{ \begin{array}{ll}
y^{2n}&\textrm {$n\equiv 0\left(mod~2\right)$}\\
y^{2n-2}y_1^2&\textrm {$n\equiv 1\left(mod~2\right)$}\\
\end{array}\right..
\end{aligned}$}
\end{eqnarray*}

Then we carry out the following operations on the polynomials $F_{k}(n)$, $G_{k}(n)$ and $H_{k}(n)$.

First, split $F_{k}(n)$ into two polynomials. One polynomial is constructed by extracting all the terms in $F_{k}(n)$ with positive coefficients denoted by $F_{k}^{+}$, and another polynomial is constructed by extracting all the terms in $F_{k}(n)$ with negative coefficients denoted by $F_{k}^{-}$. We substitute $jwi^{2m} y^{2n}$ with $j_1w_1I_1^{2m} Y_1^{2n}$ for $F_{k}^{+}$ and $j_2w_2I_2^{2m} Y_2^{2n}$ for $F_{k}^{-}$. Then we designate the resulting polynomials as $F_{k1}(n)$ and $-F_{k2}(n)$ respectively.

We can briefly describe the above operation as the following flow chart,
\begin{eqnarray*}
-F_{k2}(n)\Longleftarrow j_2w_2I_2^{2m} Y_2^{2n}\stackrel{F_{k}^{-}}{\longleftarrow}jwi^{2m} y^{2n}\stackrel{F_{k}^{+}}{\longrightarrow}j_1w_1I_1^{2m} Y_1^{2n}\Longrightarrow F_{k1}(n)
\end{eqnarray*}%

For $G_{k}(n)$ and $H_{k}(n)$, we make the same operation above as the following flow charts, respectively,
\begin{align*}
-G_{k2}(n)\Longleftarrow i_2y_2J_2^{2m} W_2^{2n}\stackrel{G_{k}^{-}}{\longleftarrow}&iyj^{2m} w^{2n}\stackrel{G_{k}^{+}}{\longrightarrow}i_1y_1J_1^{2m} W_1^{2n}\Longrightarrow G_{k1}(n)\\[9pt]
-H_{k2}(n)\Longleftarrow i_2j_2w_2y_2\stackrel{H_{k}^{-}}{\longleftarrow}&ijwy\stackrel{H_{k}^{+}}{\longrightarrow}i_1j_1w_1y_1\Longrightarrow H_{k1}(n)
\end{align*}


Combining \eqref{Eq f,g,h} and \eqref{Ineq w,y,i,j} yields that
\begin{eqnarray}\label{Ineq f,g,h}
\begin{split}
	F_{k}(n)&>F_{k1}(n)-F_{k2}(n),&~~~~~~~
    G_{k}(n)&>G_{k1}(n)-G_{k2}(n),\\[9pt]
    H_{k}(n)&>H_{k1}(n)-H_{k2}(n).&
    \end{split}
\end{eqnarray}

Now we are in a position to estimate $e^{\sqrt{\alpha_{k}}\left(i^2+y^2-2z^2\right)}$ and $e^{\sqrt{\alpha_{k}}\left(j^2+w^2-2z^2\right)}$. By \eqref{Ineq w,y,i,j}, one can obtain that for $n\geq 122$,
\begin{eqnarray}
\begin{split}
	i_1^2+y_1^2-2z^2&<i^2+y^2-2z^2<i_2^2+y_2^2-2z^2,\\[9pt]
    j_1^2+w_1^2-2z^2&<j^2+w^2-2z^2<j_2^2+w_2^2-2z^2,
    \end{split}
\end{eqnarray}
    which implies that
\begin{eqnarray}\label{Ineq e}
\begin{split}
	e^{\sqrt{\alpha_{k}}\left(i_1^2+y_1^2-2z^2\right)}&<e^{\sqrt{\alpha_{k}}\left(i^2+y^2-2z^2\right)}<e^{\sqrt{\alpha_{k}}\left(i_2^2+y_2^2-2z^2\right)},\\[9pt]
    e^{\sqrt{\alpha_{k}}\left(j_1^2+w_1^2-2z^2\right)}&<e^{\sqrt{\alpha_{k}}\left(j^2+w^2-2z^2\right)}<e^{\sqrt{\alpha_{k}}\left(j_2^2+w_2^2-2z^2\right)}.
    \end{split}
\end{eqnarray}

In order to give a feasible bound for $e^{\sqrt{\alpha_{k}}\left(i^2+y^2-2z^2\right)}$ and $e^{\sqrt{\alpha_{k}}\left(j^2+w^2-2z^2\right)}$, we define
\begin{equation}\label{Eq phi-Phi}
  \begin{split}
\Phi(t)&=1+t +\frac{t ^2}{2}+\frac{t ^3}{6}+\frac{t ^4}{24},\\[9pt]
\phi(t)&=1+t +\frac{t ^2}{2}+\frac{t ^3}{6}+\frac{t ^4}{24}+\frac{t ^5}{120}.
   \end{split}
\end{equation}

It can be checked that for $t<0$,
\begin{equation}\label{Phi-phi}
\phi(t)<e^{t}<\Phi(t).
\end{equation}

  To apply this result to \eqref{Ineq e}, it suffices to show that both $i_2^2+y_2^2-2z^2$ and $j_2^2+w_2^2-2z^2$ are negative. 
Since
  \begin{align*}
  \resizebox{0.99\hsize}{!}{$\begin{aligned}
    &i_2^2+y_2^2-2z^2=\frac{5929\pi^{20}+75460\pi^{16}z^{8}+911232\pi^{12}z^{16}-3317760\pi^{8}z^{24}-23887872\pi^{4}z^{32}}{214990848z^{38}},\\[9pt]
    &j_2^2+w_2^2-2z^2=\frac{5929\pi^{20}+18865\pi^{16}z^{8}+56952\pi^{12}z^{16}-51840\pi^{8}z^{24}-93312\pi^{4}z^{32}}{209952z^{38}},
  \end{aligned}$}  
 \end{align*}
both $i_2^2+y_2^2-2z^2$ and $j_2^2+w_2^2-2z^2$ are negative for all $n\geq 1$ can be verified by a direct calculation. Thus, applying \eqref{Phi-phi} to \eqref{Ineq e} , we see that for $n\geq 122$,
\begin{eqnarray}\label{Ineq Phi-e-phi}
 \begin{split}
	\phi \left(\sqrt{\alpha_{k}}\left(i_1^2+y_1^2-2z^2\right)\right)&<e^{\sqrt{\alpha_{k}}\left(i^2+y^2-2z^2\right)}<\Phi \left(\sqrt{\alpha_{k}}\left(i_2^2+y_2^2-2z^2\right)\right),\\[9pt]
    \phi \left(\sqrt{\alpha_{k}}\left(j_1^2+w_1^2-2z^2\right)\right)&<e^{\sqrt{\alpha_{k}}\left(j^2+w^2-2z^2\right)}<\Phi \left(\sqrt{\alpha_{k}}\left(j_2^2+w_2^2-2z^2\right)\right).
    \end{split}
 \end{eqnarray}

Now, we will proceed to prove \eqref{key Eq}. For convenience, let
\begin{equation}
A(z)=h_1e^{\sqrt{\alpha_{k}}\left(i^2+y^2-2z^2\right)}+h_2e^{\sqrt{\alpha_{k}}\left(j^2+w^2-2z^2\right)}+h_3,
\end{equation}
we need to show the positivity of $A(z)$. Applying \eqref{Eq h_i}, \eqref{Ineq f,g,h} and \eqref{Ineq Phi-e-phi}, we can obtain that for $n\geq 3512$,
\begin{equation*}
  \begin{split}
  \resizebox{0.9999\hsize}{!}{$\begin{aligned}
A(z)&>\left[F_{k1}(n)\phi\left(\sqrt{\alpha_{k}}\left(i_1^2+y_1^2-2z^2\right)\right)-F_{k2}(n)\Phi \left(\sqrt{\alpha_{k}}\left(i_2^2+y_2^2-2z^2\right)\right)\right]j^{28}w^{28}z^{58}\beta(i)\beta(y)\\[9pt]
&+\left[G_{k1}(n)\phi\left(\sqrt{\alpha_{k}}\left(j_1^2+w_1^2-2z^2\right)\right)-G_{k2}(n)
\Phi\left(\sqrt{\alpha_{k}}\left(j_2^2+w_2^2-2z^2\right)\right)\right]i^{28}y^{28}z^{58}\gamma(j)\gamma(w)\\[9pt]
&+\left[H_{k1}(n)-H_{k2}(n)\right]i^{28}j^{28}w^{28}y^{28}\gamma(z)^2.
\end{aligned}$}  
  \end{split}
\end{equation*}
Denote the right-hand side of the above inequality by $A_1(z)$. Substituting
\eqref{beta}, \eqref{Eq f,g,h}, \eqref{Eq w,y,i,j}, \eqref{w_j} and \eqref{Eq phi-Phi} into $A_1(z)$, by Mathematica we can rewrite $A_1(z)$ as
\begin{equation}
A_1(z)={\sum_{j=0}^{193} a_{j}(k) z^{2j}\over 2^{99} 3^{87} 5 z^{234}},
\end{equation}
where $a_k$ are known real numbers, and the first three terms $a_{191}(k)$, $a_{192}(k)$, $a_{193}(k)$ are listed below
\begin{equation*}
 \begin{split}
a_{191}(k)&=2^{131} 3^{86} 5^2 \alpha_{k}^2\left(-3528 -792\alpha_{k}^3+77\pi ^{8}\alpha_{k}^4\right),\\[9pt]
a_{192}(k)&=-2^{134} 3^{84} 5^2 \alpha_{k}^{\frac{5}{2}}\left(-9072-648\alpha_{k}^3+13\pi ^{8}\alpha_{k}^4\right),\\[9pt]
a_{193}(k)&=2^{136} 3^{84} 5 \alpha_{k}^3\left(-6048-216\alpha_{k}^3+\pi ^{8}\alpha_{k}^4\right).
 \end{split}
\end{equation*}
In addition, we obtain that $a_{191}(k)>0$, $a_{192}(k)<0$ and $a_{193}(k)>0$ for $\alpha_{k}=\frac{5k+2}{2k+1}$ when $k=1$ or $2$ by Mathematica.
Thus, for $n \geq 3512$, we have
\begin{equation}\label{Ineq A}
A(z)>{\sum_{j=0}^{193} a_{j}(k) z^{2j}\over 2^{99} 3^{87} 5 z^{234}}.
\end{equation}
Clearly,
\begin{equation}\label{Ineq A1}
\sum_{j=0}^{193} a_{j}(k) z^{2j}>\sum_{j=0}^{191}-|a_{j}(k)|z^{2k}+a_{192}z^{384}+a_{193}z^{386}.
\end{equation}
Moreover, it can be readily checked that for any $0\leq j \leq 190 $ and $ n\geq 5$,
 \begin{equation}
-|a_{j}(k)|z^{2j}>-a_{191}(k)z^{382}.
\end{equation}
It follows that for $ n\geq 5$,
 \begin{equation}\label{Ineq A1}
\sum_{j=0}^{193}a_{j}(k)z^{2j}>\left(-192a_{191}(k)+a_{192}(k)z^{2}+a_{193}(k)z^{4}\right)z^{382}.
\end{equation}
One can verify that for $z> 10.1$, or alternatively for $n\geq 1578$,
\begin{equation}
-192a_{191}(k)+a_{192}(k)z^{2}+a_{193}(k)z^{4}>0.
\end{equation}
Assembling all these results above yields that \eqref{key proof} is true for $ n\geq 3512$. Inequality \eqref{Lag Ineq q(n)} is equivalently true for $ n\geq 3512$. 

 On the other hand, numerical evidence shows that \eqref{Lag Ineq q(n)} is also true for $12\leq n \leq 3512$. The proof is completed.
\qed
\end{proof}

\section{A proof for Theorem \ref{3jie-theorem}}\label{4}
In this section, we will apply the approach in Section \ref{3} to prove the broken $k$-diamond partition function $\Delta_{k}(n)$ satisfies the determinantal inequalities of order $3$ with the aid of Theorem \ref{generalize} and Lemma \ref{B_I}. Note that we may use notions which have been used before but with different meanings.
\begin{theo}\label{3jie-theorem}
Let $\Delta_{k}(n)$ denote the broken $k$-diamond partition function. For $n\geq 18$ and $k=1$ or $2$, we have
	\begin{eqnarray}\label{2Lag Ineq q(n)}
   \left|\begin{array}{cccc}
		\Delta_{k}(n+2) & \Delta_{k}(n+3) &  \Delta_{k}(n+4)\\
		\Delta_{k}(n+1) & \Delta_{k}(n+2) & \Delta_{k}(n+3)\\
		\Delta_{k}(n) & \Delta_{k}(n+1) & \Delta_{k}(n+2)\\
		\end{array} \right|>0.
	\end{eqnarray}
\end{theo}
\begin{proof}
Note that for  $n\geq 3512$, \eqref{Generalize} can be rewritten as
 \begin{eqnarray}\label{Ineq q(n)2}
	\frac{\alpha_{k}\pi^{3}}{18u^{24}}I_2\left(\sqrt{\alpha_{k}}u^2\right)\left({u}^{20}-1\right)\leq \Delta_{k} \leq \frac{\alpha_{k}\pi^{3}}{18u^{24}}I_2\left(\sqrt{\alpha_{k}}u^2\right)\left({u}^{20}+1\right),
	\end{eqnarray}

  For convenience, we denote
\begin{eqnarray}\label{alpha}
	\beta(t)=t^{20}+1,~~\gamma(t)=t^{20}-1,
	\end{eqnarray}
 and
   \begin{eqnarray}
	\tilde{f}_{k}(n)\colon =\frac{\alpha_{k}\pi^{3}}{18u^{24}}\gamma(u)I_{23}\left(\sqrt{\alpha_{k}}u^2\right),\\[9pt]
    \tilde{g}_{k}(n)\colon =\frac{\alpha_{k}\pi^{3}}{18u^{24}}\beta(u)I_{24}\left(\sqrt{\alpha_{k}}u^2\right).
     \end{eqnarray}
     
 Since $I_{23}(s)\leq I_2(s) \leq I_{24}(s)$ holds for $s\geq 50$ as shown in Lemma \ref{B_I}, then we have $\tilde{f}_{k}(n) \leq \Delta_{k}(n) \leq \tilde{g}_{k}(n)$ for $n\geq 3512$.
    Thus, in order to prove Theorem \ref{3jie-theorem}, it is sufficient to prove that 
\begin{eqnarray*}
   \begin{split}
			&\tilde{f}_{k}(n+2)^3+\tilde{f}_{k}(n)\tilde{f}_{k}(n+3)^2+\tilde{f}_{k}(n+1)^2\tilde{f}_{k}(n+4)\\[9pt]
                &-\tilde{g}_{k}(n)\tilde{g}_{k}(n+2)\tilde{g}_{k}(n+4)-2\tilde{g}_{k}(n+1)\tilde{g}_{k}(n+2)\tilde{g}_{k}(n+3)>0,
   \end{split}
	\end{eqnarray*}
 which can be rewritten as
  \begin{eqnarray}\label{2key1}
   \begin{split}
			&1+\frac{\tilde{f}_{k}(n)\tilde{f}_{k}(n+3)^2}{\tilde{f}_{k}(n+2)^3}+\frac{\tilde{f}_{k}(n+1)^2\tilde{f}_{k}(n+4)}{\tilde{f}_{k}(n+2)^3}\\[9pt]
                &-\frac{\tilde{g}_{k}(n)\tilde{g}_{k}(n+2)\tilde{g}_{k}(n+4)}{\tilde{f}_{k}(n+2)^3}-\frac{2\tilde{g}_{k}(n+1)\tilde{g}_{k}(n+2)\tilde{g}_{k}(n+3)}{\tilde{f}_{k}(n+2)^3}>0.
   \end{split}
	\end{eqnarray}
 Using the same notation \eqref{Eq u} as in Section \ref{3}, we can rewrite the left-hand side of the above inequality as
\begin{align}\label{2Eq key}
	\frac{1}{h_5}&\left(h_1e^{\sqrt{\alpha_k}\left(j^2+w^2-2z^2\right)}+h_2e^{\sqrt{\alpha_k}\left(i^2+y^2-2z^2\right)}
\notag\right.
\\
&
\left.
+h_3e^{\sqrt{\alpha_k}\left(2i^2+w^2-3z^2\right)}+h_4e^{\sqrt{\alpha_k}\left(j^2+2y^2-3z^2\right)}+h_5\right),
	\end{align}
where
\begin{eqnarray}\label{2Eq h_i}
   \begin{split}
	h_1&=F_{k}(n)i^{96}y^{96}z^{98}\beta(j)\beta(w)\beta(z),\\[9pt]
    h_2&=G_{k}(n)i^{48}j^{48}w^{48}y^{48}z^{98}\beta(i)\beta(y)\beta(z),\\[9pt]
    h_3&=H_{k}(n)j^{48}y^{96}z^{147}\gamma(i)^2\gamma(w),\\[9pt]
    h_4&=L_{k}(n)i^{96}w^{48}z^{147}\gamma(j)\gamma(y)^2,\\[9pt]
    h_5&=M_{k}(n)i^{96}j^{48}w^{48}y^{96}\gamma(z)^3,
     \end{split}
\end{eqnarray}
and 
\begin{small} 
\begin{align}\label{F,M}
	&F_{k}(n)=\sum_{m,n,t=0}^{12}a_{m,n}(k)i^2y^2j^{2m} w^{2n}z^{2t},&~
    &G_{k}(n)=\sum_{m,n,t=0}^{12}b_{m,n}(k)jwi^{2m+1} y^{2n+1}z^{2t},\notag\\[9pt]
    &H_{k}(n)=\sum_{m=0}^{24}\sum_{n=0}^{12}c_{m,n}(k)jy^2i^{2m} w^{2n},&~
    &L_{k}(n)=\sum_{m=0}^{12}\sum_{n=0}^{24}d_{m,n}(k)i^2wj^{2m} y^{2n},\notag\\[9pt]
    &M_{k}(n)=\sum_{t=0}^{36}l_{t}(k)i^2jwy^2z^{2t},&
\end{align}
\end{small} 
where $a_{m,n}(k)$, $b_{m,n}(k)$, $c_{m,n}(k)$, $d_{m,n}(k)$ and $l_{t}(k)$ are related to $k$ and are the real coefficients of polynomials $F_{k}(n)$, $G_{k}(n)$, $H_{k}(n)$, $L_{k}(n)$ and $M_{k}(n)$ respectively.

Now we proceed to prove \eqref{2Eq key} is positive for $n\geq3512$. Applying \eqref{u} and \eqref{Eq u} into the expression of  $M_{k}(n)$, one can easily deduce that $M_{k}(n)>0$ for all $n\geq 1$ by Mathematica, which implies that the denominator $h_5$ of \eqref{2Eq key} is positive for $n\geq 1$. 
Thus, all that is required of us is to show
\begin{align}\label{key proof}
	&h_1e^{\sqrt{\alpha_k}\left(j^2+w^2-2z^2\right)}+h_2e^{\sqrt{\alpha_k}\left(i^2+y^2-2z^2\right)}\notag\\[9pt]
&+h_3e^{\sqrt{\alpha_k}\left(2i^2+w^2-3z^2\right)}+h_4e^{\sqrt{\alpha_k}\left(j^2+2y^2-3z^2\right)}+h_5>0.
	\end{align}

In order to accomplish it, we need to estimate $h_1$, $h_2$, $h_3$, $h_4$, $h_5$, $e^{\sqrt{\alpha_k}\left(j^2+w^2-2z^2\right)}$, $e^{\sqrt{\alpha_k}\left(i^2+y^2-2z^2\right)}$, $e^{\sqrt{\alpha_k}\left(2i^2+w^2-3z^2\right)}$, and $e^{\sqrt{\alpha_k}\left(j^2+2y^2-3z^2\right)}$. 
We employ the notation \eqref{Eq w,y,i,j} of $w$, $y$, $i$ and $j$ as in the proof of Theorem \ref{theorem1} for $n\geq 1$, and use the first few terms until $z^{-27}$ of their Taylor expansion to get that for $n\geq 2404$,
\begin{eqnarray}\label{Ineq W,j}
 \begin{split}
    w_1&<w<w_2,&~~~~~~
    y_1&<y<y_2,\\[9pt]
    i_1&<i<i_2,&
    j_1&<j<j_2,
 \end{split}
\end{eqnarray}
where
\begin{align} \label{x_i}
    w_1&=z-\frac{\pi^{2}}{3z^{3}}-\frac{\pi^{4}}{6z^{7}}-\frac{7\pi^{6}}{54z^{11}}-\frac{77\pi^{8}}{648z^{15}}-\frac{77\pi^{10}}{468z^{19}}
    -\frac{61\pi^{12}}{486z^{23}},\notag\\[9pt]
    w_2&=z-\frac{\pi^{2}}{3z^{3}}-\frac{\pi^{4}}{6z^{7}}-\frac{7\pi^{6}}{54z^{11}}-\frac{77\pi^{8}}{648z^{15}}-\frac{77\pi^{10}}{648z^{19}}
    -\frac{1463\pi^{12}}{11664z^{23}},\notag\\[9pt]
   y_1&=z-\frac{\pi^{2}}{6z^{3}}-\frac{\pi^{4}}{24z^{7}}-\frac{7\pi^{6}}{432z^{11}}-\frac{77\pi^{8}}{10368z^{15}}-\frac{13\pi^{10}}{3456z^{19}} -\frac{61\pi^{12}}{31104z^{23}},\\[9pt]
   y_2&=z-\frac{\pi^{2}}{6z^{3}}-\frac{\pi^{4}}{24z^{7}}-\frac{7\pi^{6}}{432z^{11}}-\frac{77\pi^{8}}{10368z^{15}}-\frac{77\pi^{10}}{20736z^{19}}
     -\frac{1463\pi^{12}}{746496z^{23}},\notag\\[9pt]
    i_1&=z+\frac{\pi^{2}}{6z^{3}}-\frac{\pi^{4}}{24z^{7}}+\frac{7\pi^{6}}{432z^{11}}-\frac{77\pi^{8}}{10368z^{15}}+\frac{77\pi^{10}}{20736z^{19}}
    -\frac{1463\pi^{12}}{746496z^{23}},\notag\\[9pt]
    i_2&=z+\frac{\pi^{2}}{6z^{3}}-\frac{\pi^{4}}{24z^{7}}+\frac{7\pi^{6}}{432z^{11}}-\frac{77\pi^{8}}{10368z^{15}}+\frac{77\pi^{10}}{20736z^{19}},\notag\\[9pt]
    j_1&=z+\frac{\pi^{2}}{3z^{3}}-\frac{\pi^{4}}{6z^{7}}+\frac{7\pi^{6}}{54z^{11}}-\frac{77\pi^{8}}{648z^{15}}
    +\frac{77\pi^{10}}{648z^{19}}-\frac{1463\pi^{12}}{11664z^{23}},\notag\\[9pt]
    j_2&=z+\frac{\pi^{2}}{3z^{3}}-\frac{\pi^{4}}{6z^{7}}+\frac{7\pi^{6}}{54z^{11}}-\frac{77\pi^{8}}{648z^{15}}+\frac{77\pi^{12}}{648z^{19}}.\notag
\end{align}
Next we turn to estimate $h_1$, $h_2$, $h_3$, $h_4$ and $h_5$, which indicates to estimate $F_{k}(n)$, $G_{k}(n)$, $H_{k}(n)$, $L_{k}(n)$ and $M_{k}(n)$. Due to this, we will continue with the equation \eqref{J} from Section \ref{3} and add the following two equations,
\begin{eqnarray*}
I_1^{2m+1}=\left\{ \begin{array}{ll}
i^{2m}i_1 &\textrm {$ m\equiv 0\left(mod~2\right)$}\\
i^{2m-2}i_1^3 &\textrm {$ m\equiv 1\left(mod~2\right)$}\\
\end{array}\right.,
I_2^{2m+1}=\left\{ \begin{array}{ll}
i^{2m}i_2&\textrm {$m\equiv 0\left(mod~2\right)$}\\
i^{2m-2}i_2^3&\textrm {$m\equiv 1\left(mod~2\right)$}\\
\end{array}\right..
\end{eqnarray*}
Then, to estimate $F_{k}(n)$, $G_{k}(n)$, $H_{k}(n)$, $L_{k}(n)$ and $M_{k}(n)$, we will make the same substitution as in Section \ref{3} illustrated in the following flow charts,
\begin{align*}
-F_{k2}(n)\Longleftarrow i_2^2y_2^2J_2^{2m} W_2^{2n}\stackrel{F_{k}^{-}}{\longleftarrow}&i^2y^2j^{2m} w^{2n}\stackrel{F_{k}^{+}}{\longrightarrow}i_1^2y_1^2J_1^{2m} W_1^{2n}\Longrightarrow F_{k1}(n)\\[9pt]
-G_{k2}(n)\Longleftarrow j_2w_2I_2^{2m+1}Y_2^{2n+1}\stackrel{G_{k}^{-}}{\longleftarrow}&jwi^{2m+1} y^{2n+1}\stackrel{G_{k}^{+}}{\longrightarrow}j_1w_1I_1^{2m+1}Y_1^{2n+1}\Longrightarrow G_{k1}(n)\\[9pt]
-H_{k2}(n)\Longleftarrow j_2y_2^2I_2^{2m}W_2^{2n}\stackrel{H_{k}^{-}}{\longleftarrow}&jy^2i^{2m} w^{2n}\stackrel{H_{k}^{+}}{\longrightarrow}j_1y_1^2I_1^{2m}W_1^{2n}\Longrightarrow H_{k1}(n)\\[9pt]
-L_{k2}(n)\Longleftarrow i_2^2w_2J_2^{2m}Y_2^{2n}\stackrel{L_{k}^{-}}{\longleftarrow}&i^2wj^{2m} y^{2n}\stackrel{L_{k}^{+}}{\longrightarrow}i_1^2w_1J_1^{2m}Y_1^{2n}\Longrightarrow L_{k1}(n)\\[9pt]
-M_{k2}(n)\Longleftarrow i_2^2j_2w_2y_2^2\stackrel{M_{k}^{-}}{\longleftarrow}&i^2jwy^2\stackrel{M_{k}^{+}}{\longrightarrow}i_1^2j_1w_1y_1^2\Longrightarrow M_{k1}(n)
\end{align*}





Combining \eqref{F,M} and \eqref{Ineq W,j} leads to
\begin{eqnarray}\label{Ineq f,g,h,L,M}
\begin{split}
	F_{k}(n)&>F_{k1}(n)-F_{k2}(n),&~~~~~~
    G_{k}(n)&>G_{k1}(n)-G_{k2}(n),\\[9pt]
    H_{k}(n)&>H_{k1}(n)-H_{k2}(n),&~~~~~~
    L_{k}(n)&>L_{k1}(n)-L_{k2}(n),\\[9pt]
    M_{k}(n)&>M_{k1}(n)-M_{k2}(n).
    \end{split}
\end{eqnarray}

Now we proceed to estimate $e^{\sqrt{\alpha_k}\left(j^2+w^2-2z^2\right)}$, $e^{\sqrt{\alpha_k}\left(i^2+y^2-2z^2\right)}$, $e^{\sqrt{\alpha_k}\left(2i^2+w^2-3z^2\right)}$ and $e^{\sqrt{\alpha_k}\left(j^2+2y^2-3z^2\right)}$. By \eqref{Ineq W,j}, one can obtain that for $n\geq 2404$,
\begin{eqnarray}
\begin{split}
    j_1^2+w_1^2-2z^2&<j^2+w^2-2z^2<j_2^2+w_2^2-2z^2,\\[9pt]
    i_1^2+y_1^2-2z^2&<i^2+y^2-2z^2<i_2^2+y_2^2-2z^2,\\[9pt]
    2i_1^2+w_1^2-3z^2&<2i^2+w^2-3z^2<2i_2^2+w_2^2-3z^2,\\[9pt]
    j_1^2+2y_1^2-3z^2&<j^2+2y^2-3z^2<j_2^2+2y_2^2-3z^2,
    \end{split}
\end{eqnarray}
    which implies that
\begin{eqnarray}\label{2Ineq e}
\begin{split}
    e^{\sqrt{\alpha_k}\left(j_1^2+w_1^2-2z^2\right)}&<e^{\sqrt{\alpha_k}\left(j^2+w^2-2z^2\right)}<e^{\sqrt{\alpha_k}\left(j_2^2+w_2^2-2z^2\right)},\\[9pt]
    e^{\sqrt{\alpha_k}\left(i_1^2+y_1^2-2z^2\right)}&<e^{\sqrt{\alpha_k}\left(i^2+y^2-2z^2\right)}<e^{\sqrt{\alpha_k}\left(i_2^2+y_2^2-2z^2\right)},\\[9pt]
	e^{\sqrt{\alpha_k}\left(2i_1^2+w_1^2-3z^2\right)}&<e^{\sqrt{\alpha_k}\left(2i^2+w^2-3z^2\right)}<e^{\sqrt{\alpha_k}\left(2i_2^2+w_2^2-3z^2\right)},\\[9pt]
    e^{\sqrt{\alpha_k}\left(j_1^2+2y_1^2-3z^2\right)}&<e^{\sqrt{\alpha_k}\left(j^2+2y^2-3z^2\right)}<e^{\sqrt{\alpha_k}\left(j_2^2+2y_2^2-3z^2\right)}.
    \end{split}
\end{eqnarray}
Moreover, a direct calculation by Mathematica leads to $j_2^2+w_2^2-2z^2$, $i_2^2+y_2^2-2z^2$, $2i^2+w^2-3z^2$ and $j^2+2y^2-3z^2$ are all negative for all $n\geq 1$. Here we omit their tedious expansions. Thus, applying \eqref{Phi-phi} to \eqref{2Ineq e} , we get that for $n\geq 2404$,
\begin{eqnarray}\label{2Ineq Phi-e-phi}
 \begin{split}
 \phi \left(\sqrt{\alpha_k}\left(j_1^2+w_1^2-2z^2\right)\right)&<e^{\sqrt{\alpha_k}\left(j^2+w^2-2z^2\right)}<\Phi \left(\sqrt{\alpha_k}\left(j_2^2+w_2^2-2z^2\right)\right),\\[9pt]
    \phi \left(\sqrt{\alpha_k}\left(i_1^2+y_1^2-2z^2\right)\right)&<e^{\sqrt{\alpha_k}\left(i^2+y^2-2z^2\right)}<\Phi \left(\sqrt{\alpha_k}\left(i_2^2+y_2^2-2z^2\right)\right),\\[9pt]
	\phi \left(\sqrt{\alpha_k}\left(2i_1^2+w_1^2-3z^2\right)\right)&<e^{\sqrt{\alpha_k}\left(2i^2+w^2-3z^2\right)}<\Phi \left(\sqrt{\alpha_k}\left(2i_2^2+w_2^2-3z^2\right)\right),\\[9pt]
    \phi \left(\sqrt{\alpha_k}\left(j_1^2+2y_1^2-3z^2\right)\right)&<e^{\sqrt{\alpha_k}\left(j^2+2y^2-3z^2\right)}<\Phi \left(\sqrt{\alpha_k}\left(j_2^2+2y_2^2-3z^2\right)\right).
    \end{split}
 \end{eqnarray}

Now, we proceed to prove \eqref{key proof}. With similar arguments used in Section \ref{3}, we let
\begin{align}
A(z)=&h_1e^{\sqrt{\alpha_k}\left(j^2+w^2-2z^2\right)}+h_2e^{\sqrt{\alpha_k}\left(i^2+y^2-2z^2\right)}\notag\\[9pt]
&+h_3e^{\sqrt{\alpha_k}\left(2i^2+w^2-3z^2\right)}+h_4e^{\sqrt{\alpha_k}\left(j^2+2y^2-3z^2\right)}+h_5,
\end{align}
we need to show the positivity of $A(z)$. By employing \eqref{F,M}, \eqref{Ineq f,g,h,L,M} and \eqref{2Ineq Phi-e-phi}, we get that for $n\geq 3512$,
\begin{equation*}
  \begin{split}
  \resizebox{0.9999\hsize}{!}{$\begin{aligned}
A(z)&>\left[F_{k1}(n)\phi\left(\sqrt{\alpha_k} \left(j_1^2+w_1^2-2z^2\right)\right)-F_{k2}(n)\Phi\left(\sqrt{\alpha_k} \left(j_2^2+w_2^2-2z^2\right)\right)\right]i^{96}y^{96}z^{98}\beta(j)\beta(w)\beta(z)\\[9pt]
&+\left[G_{k1}(n) \phi\left(\sqrt{\alpha_k} \left(i_1^2+y_1^2-2z^2\right)\right)-G_{k2}(n)
\Phi\left(\sqrt{\alpha_k} \left(i_2^2+y_2^2-2z^2\right)\right)\right]i^{48}j^{48}w^{48}y^{48}z^{98}\beta(i)\beta(y)\beta(z)\\[9pt]
&+\left[H_{k1}(n)\phi\left(\sqrt{\alpha_k}\left(2i_1^2+w_1^2-3z^2\right)\right)-H_{k2}(n)\Phi\left(\sqrt{\alpha_k} \left(2i_2^2+w_2^2-3z^2\right)\right)\right]j^{48}y^{96}z^{147}\gamma(i)^2\gamma(w)\\[9pt]
&+\left[L_{k1}(n)\phi\left(\sqrt{\alpha_k}\left(j_1^2+2y_1^2-3z^2\right)\right)-L_{k2}(n)
\Phi\left(\sqrt{\alpha_k} \left(j_2^2+2y_2^2-3z^2\right)\right)\right]i^{96}w^{48}z^{147}\gamma(j)\gamma(y)^2\\[9pt]
&+\left[M_{k1}(n)-M_2(n)\right]i^{96}j^{48}w^{48}y^{96}\gamma(z)^3.
\end{aligned}$}  
  \end{split}
\end{equation*}
Denote the right-hand side of the above inequality as $A_1(z)$. Applying \eqref{Eq w,y,i,j}, \eqref{Eq phi-Phi}, \eqref{alpha}, \eqref{F,M}  and \eqref{x_i} to $A_1(z)$, by Mathematica, we can restate $A_1(z)$ as
\begin{equation}
A_1(z)={\sum_{j=0}^{362} a_j(k) z^{2j}\over 2^{143} 3^{169} 5z^{316}},
\end{equation}
where $a_j(k)$ are known real numbers, and the first three terms $a_{360}(k)$, $a_{361}(k)$, $a_{362}(k)$ are listed below,
\begin{equation*}
 \begin{split}
a_{360}(k)&=2^{260} 3^{165} 5 \alpha_{k}^{\frac{35}{2}}\left(524880+2123\pi ^{12}\alpha_{k}\right),\\[9pt]
a_{361}(k)&=-2^{264} 3^{166} 5 \alpha_{k}^{18}\left(1944+7\pi ^{12}\alpha_{k}\right),\\[9pt]
a_{362}(k)&=2^{267} 3^{163} 5 \pi ^{12}\alpha_{k}^{\frac{39}{2}}.
 \end{split}
\end{equation*}
It is clear that $a_{360}(k)>0$, $a_{361}(k)<0$ and $a_{362}(k)>0$ for all $\alpha_{k}=\frac{5k+2}{2k+1}>0$.
Thus, for $n \geq 3512$, we have
\begin{equation}\label{Ineq A}
A(z)>{\sum_{j=0}^{362} a_j(k) z^{2j}\over 2^{143} 3^{169} 5z^{316}}.
\end{equation}
Obviously,
\begin{equation}\label{2Ineq A1}
\sum_{j=0}^{362}a_j(k)z^{2j}>\sum_{j=0}^{360}-|a_j(k)|z^{2j}+a_{361}z^{722}+a_{362}z^{724}.
\end{equation}
Moreover, it can be readily checked that for any $0\leq j \leq 360 $ and $ n\geq 10$,
 \begin{equation}
-|a_j(k)|z^{2j}>-a_{360}z^{720},
\end{equation}
it follows that for $ n\geq 10$,
 \begin{equation}\label{Ineq A1}
\sum_{j=0}^{360}a_j(k)z^{2j}>\left(-361a_{360}(k)+a_{361}(k)z^{2}+a_{362}(k)z^{4}\right)z^{720}.
\end{equation}
One can check that for $z> 12.65$, or equivalently, for $n\geq 3891$,
\begin{equation}
-361a_{360}(k)+a_{361}(k)z^{2}+a_{362}(k)z^{4}>0.
\end{equation}
Combining all these results above reveals that inequality \eqref{key proof} holds for $ n\geq 3891$, which implies that inequality \eqref{2Lag Ineq q(n)} holds for $ n\geq 3891$.

 On the other hand, numerical evidence shows that \eqref{2Lag Ineq q(n)} also holds for $18\leq n \leq 3891$. The proof is completed.
\qed
\end{proof}
\section{Open problems}\label{5}
The results above encourage us to consider the Laguerre inequality of any order and the determinantal inequality of any order for $\Delta_{k}(n)$. Numerical evidence suggests us to propose the following conjectures.
\begin{conj}
Let $a_n=\Delta_{k}(n)$, then for $n\geq 14$ and $k=1$ or $2$
\begin{align*}
I(a_n,a_{n+1},a_{n+2},a_{n+3},a_{n+4})&=A(a_n,a_{n+1},a_{n+2},a_{n+3},a_{n+4})^3\\[9pt]
&-27B(a_n,a_{n+1},a_{n+2},a_{n+3},a_{n+4})^2>0,\\[9pt]
\text{where}\quad A(a_n,a_{n+1},a_{n+2},a_{n+3},a_{n+4})&=a_{n}a_{n+4}-4a_{n+1}a_{n+3}+3a_{n+2}^2,\\[9pt]
B(a_n,a_{n+1},a_{n+2},a_{n+3},a_{n+4})&=-a_{n}a_{n+2}a_{n+4}+a_{n+2}^3+a_{n}a_{n+3}^2\\[9pt]
&+a_{n+1}^2a_{n+4}-2a_{n+1}a_{n+2}a_{n+3}.
\end{align*}
\end{conj}
Recall that $A>0$, $B>0$ are equivalent to Theorem \ref{theorem1} and Theorem \ref{3jie-theorem}, respectively.
\begin{conj}
For $k=1$ or $2$ and $1 \leq m \leq 14$, the broken $k$-diamond partition function $\Delta_{k}(n)$ satisfies the Laguerre inequality of order $m$ for $n \geq N_{\Delta_{k}}(m)$, where
\begin{table}[H]
  \centering
   \renewcommand{\arraystretch}{1.5}
    \setlength{\tabcolsep}{3pt}
\begin{tabular}{ccccccccccccccc}
\hline
$m$ & 1 & 2 & 3 & 4 & 5 & 6 & 7 & 8 & 9 & 10 & 11 & 12 & 13 & 14\\
\hline
$N_{\Delta_{1}}(m)$ & 1 & 12 & 53 & 132 & 251 & 420 & 639 & 912 & 1245 & 1636 & 2091 & 2612 & 3201 & 3858\\
\hline
$N_{\Delta_{2}}(m)$ & 1 & 10 & 45 & 106 & 211 & 354 & 539 & 774 & 1059 & 1398 & 1781 & 2240 & 2749 & 3318\\
\hline
\end{tabular}
\end{table}
\end{conj}

\begin{conj}
For $k=1$ or $2$ and $1 \leq m \leq 14$, the broken $k$-diamond partition function $\Delta_{k}(n)$ satisfies $\det (\Delta_{k}(n-i+j))_{1\leq i,j\leq m}$ for $n \geq \mathbf{N}_{\Delta_{k}}(m)$, where
\begin{table}[H]
  \centering
   \renewcommand{\arraystretch}{1.5}
    \setlength{\tabcolsep}{3.5pt}
\begin{tabular}{ccccccccccccccc}
\hline
$m$ & 1 &~ 2 & 3 & 4 & 5 & 6 & 7 & 8 & 9 & 10 & 11 & 12 & 13 & 14\\
\hline
$\mathbf{N}_{\Delta_{1}}(m)$ & 1 & 1 & 20 & 84 & 194 & 362 & 594 & 890 & 1258 & 1700 & 2218 & 2818 & 3498 & 4264\\
\hline
$\mathbf{N}_{\Delta_{2}}(m)$ & 1 & 1 & 18 & 72 & 168 & 308 & 506 & 762 & 1082 & 1464 & 1914 & 2436 & 3028 & 3696\\
\hline
\end{tabular}
\end{table}
\end{conj}

\end{document}